\documentclass[a4paper]{amsart}
\usepackage{amsmath,amsfonts,amssymb,amsthm}
\usepackage{enumerate}
\newtheorem{theorem}{Theorem}[section]
\newtheorem{lemma}[theorem]{Lemma}
\newtheorem{proposition}[theorem]{Proposition}

\theoremstyle{remark}
\newtheorem{remark}[theorem]{Remark}
\theoremstyle{definition}

\title{Almost primes in almost all very short intervals}
\author{Kaisa Matom\"aki}
\address{Department of Mathematics and Statistics, University of Turku, 20014 Turku, Finland}
\email{ksmato@utu.fi}
\begin{document}

\begin{abstract}
We show that as soon as $h\to \infty$ with $X \to \infty$, almost all intervals $(x-h\log X, x]$ with $x \in (X/2, X]$ contain a product of at most two primes. In the proof we use Richert's weighted sieve, with the arithmetic information eventually coming from results of Deshouillers and Iwaniec on averages of Kloosterman sums.
\end{abstract}
\maketitle

\section{Introduction}

By probabilistic models, one expects that short intervals of the type $(x-h \log X, x]$ contain primes for almost all $x \in (X/2, X]$ as soon as $h \to \infty$ with $X \to \infty$. Heath-Brown~\cite{H-BRHPC} has established this assuming simultaneously the Riemann hypothesis and the pair correlation conjecture for the zeros of the Riemann zeta function. Without such strong hypotheses we are rather far from this claim --- the best result~\cite{Jia} today is that almost all intervals of length $X^{1/20}$ contain primes.

One can ask a similar question about almost-primes, i.e. $P_k$ numbers that have at most $k$ prime factors or $E_k$ numbers that have exactly $k$ prime factors. In the second case the best results are due to Ter\"av\"ainen~\cite{Tera} who showed that, for any $\varepsilon > 0$, almost all intervals of length $(\log \log X)^{6+\varepsilon} \log X$ contain an $E_3$ number and that almost all intervals of length $(\log X)^{3.51}$ contain an $E_2$-number.

The case of $P_k$ numbers is significantly easier than that of $E_k$ numbers since the so called parity barrier does not apply; due to the parity barrier classical sieve methods based only on so called type I information cannot distinguish numbers having an even number of prime factors from those having an odd number of prime factors. In particular classical sieve methods can, in favourable circumstances, be used to show that a given set contains $P_2$ numbers, but to show that it contains $E_2$ numbers requires additional arithmetic information. For more information about the parity barrier, see e.g.~\cite[Section 16.4]{Opera}.

Following Friedlander~\cite{Fried1, Fried2}, Friedlander and Iwaniec~\cite{Opera} showed that as soon as $h \to \infty$ with $X\to \infty$, almost all intervals $(x-h \log X, x]$ contain $P_{19}$-numbers.

In~\cite[Between Corollary 6.28 and Proposition 6.29]{Opera} Friedlander and Iwaniec discuss the possibility to improve their result. In particular they mention that using linear sieve theory and estimates for general bilinear forms of exponential sums with Kloosterman fractions from~\cite{DFI}, one should be able to improve $P_{19}$ to $P_3$. Then they say that "It would be interesting to get integers with at most two prime divisors". This is the aim of the current note.

Let us introduce a few notational conventions before stating our main theorem: The letter $p$ with or without subscripts always denotes a prime number. We write $\Omega(n)$ for the total number of prime factors of $n$ and $\omega(n)$ for the number of distinct prime factors of $n$. Furthermore we write $\mathbf{1}_P$ for the indicator function of a claim $P$. Further notational conventions, including our asymptotic notation, are described in Section~\ref{sse:Notation}. 
\begin{theorem}
\label{th:MT}
There exists a constant $c > 0$ such that the following holds. Let $X \geq 2$ and $2 \leq h \leq X^{1/100}$. Then
\[
\sum_{\substack{x - h \log X  < n \leq x \\ p \mid n \implies p > X^{1/8}}} \mathbf{1}_{\Omega(n) \leq 2} \geq c h
\]
for all $x \in (X/2, X]$ apart from an exceptional set of measure $O(X/h)$.
\end{theorem}
Hence, as soon as $h \to \infty$ with $X \to \infty$, almost all intervals of length $h\log X$ contain $P_2$-numbers. Previously it was known, as a consequence of the work of Ter\"av\"ainen~\cite{Tera} on $E_2$ numbers, that almost all intervals of length $(\log X)^{3.51}$ contain $P_2$-numbers. Before Ter\"av\"ainen's work, the best result was due to Mikawa~\cite{Mikawa} who showed that as soon as $h \to \infty$ with $X$, almost all intervals $(x - h(\log X)^5, x]$ contain $P_2$-numbers. On the other hand, by work of Wu~\cite{Wu} it is known that the interval $(x-x^{101/232}, x]$ contains $P_2$ numbers for all sufficiently large $x$.

The corresponding upper bound
\[
\sum_{\substack{x - h \log X  < n \leq x \\ p \mid n \implies p > X^{1/8}}} \mathbf{1}_{\Omega(n) \leq 2} =O(h)
\]
for all $x \in (X/2, X]$ apart from an exceptional set of measure $O(X/h)$ follows immediately from~\cite[Corollary 6.28]{Opera}. 

We will give an outline of the proof of Theorem~\ref{th:MT} in Section~\ref{se:Outline}.

\subsection{Notation}
\label{sse:Notation}
We write $\Lambda(n)$ for the von Mangoldt function, so that
\[
\Lambda(n) = 
\begin{cases}
\log p &\text{if $n = p^k$ for some prime $p$ and positive integer $k$;} \\
0 &\text{otherwise;}
\end{cases}
\]
and $\mu(n)$ for the M\"obius function so that
\[
\mu(n) = 
\begin{cases}
(-1)^{\omega(n)} &\text{if $n$ is square-free;} \\
0 & \text{otherwise.}
\end{cases}
\] 
Furthermore we write $\varphi(n)$ for the Euler $\varphi$-function so that 
\[
\varphi(n) = \sum_{\substack{1 \leq k \leq n \\ (k, n) = 1}} 1.
\]

For $f \colon \mathbb{R} \to \mathbb{C}$ and $g\colon \mathbb{R} \to \mathbb{R}^+$, we write $f(x) = O(g(x))$ or $f(x) \ll g(x)$ if there exists a constant $C > 0$ such that $|f(x)| \leq C g(x)$ for every $x$. Similarly, when also $f$ takes positive real values, we write $f(x) \gg g(x)$ if there exist a constant $c$ such that $f(x) \geq c g(x)$ for every $x$. If there is a subscript (e.g. $O_k(g(x))$), then the implied constant is allowed to depend on the parameter(s) in the subscript.

We say that $g \colon \mathbb{R} \to \mathbb{R}$ is smooth if it has derivatives of all orders. We will constantly work with smooth compactly supported functions whose support and derivatives are bounded from above independently of our parameters tending to infinity (e.g. $X$), so that
\begin{equation}
\label{eq:gderivbound}
\frac{d^k}{dx^k} g(x) \ll_k 1 \quad \quad \text{for every $k \geq 0$,}
\end{equation}
where the implied constant depends only on $k$.

For $u \in \mathbb{C}$ we write $e(u) := e(2\pi i u)$ and, for any function 
\[
g \in \mathbb{L}^1(\mathbb{R}) := \left\{f \colon \mathbb{R} \to \mathbb{R} \colon \int_{-\infty}^\infty |f(x)| dx < \infty\right\},
\]
we denote by $\widehat{g}$ the Fourier transform
\[
\widehat{g}(\xi) = \int_{-\infty}^\infty g(x) e(-\xi x) dx.
\]
If $g$ is a smooth and compactly supported function satisfying~\eqref{eq:gderivbound}, then one obtains by repeated partial integration that
\begin{equation}
\label{eq:gFourbound}
\widehat{g}(\xi) \ll_k \frac{1}{1+|\xi|^k} \quad \text{for any $\xi \in \mathbb{R}$ and $k \geq 0$}.
\end{equation}

In summation conditions, we write $m \sim M$ for $m \in (M, 2M]$. Furthermore we write $A \asymp B$ when $A \ll B \ll A$. For $a \in \mathbb{Z}$ and $q \in \mathbb{N}$ we write $\overline{a}$ for the inverse of $a \pmod{q}$ (the modulus will be clear from the context, e.g. in $e(\frac{c\overline{u}}{v})$ the inverse is $\pmod{v}$).

\subsection{Outline of the proof}
\label{se:Outline}
In this section we provide a simplified outline of the proof of Theorem~\ref{th:MT}. 

We start by applying Richert's weighted sieve (see e.g.~\cite[Chapter 25]{Opera}) which is tailored to finding $P_k$ numbers. More precisely, writing $H = h\log X$, $z = X^{5/36}$ and $P(z) = \prod_{p < z} p$,  we show in Section~\ref{se:set-up} that, for almost all $x \in (X/2, X]$,
\begin{equation}
\label{eq:OutRich}
\begin{split}
&\sum_{\substack{x - H  < n \leq x \\ (n, P(z)) = 1}} \mathbf{1}_{\Omega(n) \leq 2} \\
&\geq \frac{1}{2}\sum_{\substack{x - H < n \leq x}} \mathbf{1}_{(n, P(z)) = 1} - \frac{1}{2} \sum_{\substack{z \leq p < 2X^{1/2}}} \left(1-\frac{\log p}{\log y}\right) \sum_{\substack{x - H < np \leq x}} \mathbf{1}_{(n, P(z)) = 1}.
\end{split}
\end{equation}
Now classical sieve methods such as the $\beta$-sieve (see e.g.~\cite[Chapter 11]{Opera}) are very suitable for finding lower and upper bounds for $\mathbf{1}_{(n, P(z)) = 1}$ --- we use a lower bound sieve in the first sum and an upper bound sieve in the second sum on the right hand side of~\eqref{eq:OutRich}. We have
\begin{equation}
\label{eq:OutSieves}
\sum_{u \mid n} \lambda_u^- \leq \mathbf{1}_{(n, P(z)) = 1} \leq \sum_{u \mid n} \lambda_{u, p}^+, 
\end{equation}
where $\lambda_u^-$ are lower bound and $\lambda_{u, p}^+$ are upper bound linear sieve weights of levels $U := X^{5/9}$ and $U/p$. In particular $\lambda_u^-$ are supported on $u \leq U$ and $\lambda_{u, p}^+$ are supported on $u \leq U/p$. For precise definitions of sieve weights we use, see Section~\ref{se:set-up}.

Combining~\eqref{eq:OutRich} and~\eqref{eq:OutSieves}, we see that for almost all $x \in (X/2, X]$, one has
\[
2\sum_{\substack{x - H  < n \leq x \\ (n, P(z)) = 1}} \mathbf{1}_{\Omega(n) \leq 2} \geq \sum_u \lambda_u^- \sum_{\substack{\substack{x - H < n \leq x \\ u \mid n}}} 1 - \sum_{\substack{z \leq p < 2X^{1/2}}} \left(1-\frac{\log p}{\log y}\right) \sum_u \lambda_{u, p}^+ \sum_{\substack{x - H < np \leq x \\ u \mid n}} 1.
\]
Writing, for $d \in \{u, up\}$, 
\[
E_d(x) = \sum_{\substack{\substack{x - H < n \leq x \\ d \mid n}}} 1 - \frac{H}{d},
\]
one obtains
\[
\begin{split}
2\sum_{\substack{x - H  < n \leq x \\ (n, P(z)) = 1}} \mathbf{1}_{\Omega(n) \leq 2}  &\geq H \sum_u \frac{\lambda_u^-}{u} - H \sum_{\substack{z \leq p < 2X^{1/2}}} \frac{1}{p} \left(1-\frac{\log p}{\log y}\right) \sum_u \frac{\lambda_{u, p}^+}{u} \\
&+\sum_u \lambda_u^- E_u(x) - \sum_{\substack{z \leq p < 2X^{1/2}}} \sum_u \lambda_{u, p}^+ \left(1-\frac{\log p}{\log y}\right) E_{pu}(x).
\end{split}
\]
On the right hand side the first line gives the main term, and a calculation using known properties of the linear sieve coefficients shows that it is $\gg h$ (see Section~\ref{se:MainTermM} for details) --- here it is important that the level of distribution $U = X^{5/9}$ is a sufficiently large power of $X$; dealing with the error terms would be substantially simpler for $U = X^{1/2-\varepsilon}$ but the main term would be negative and thus the result useless.

Consequently Theorem~\ref{th:MT} follows once we have shown that
\begin{equation}
\label{eq:OutClaimE1}
\int_{X/2}^X \left|\sum_u \lambda_u^- E_u(x)\right|^2 dx \ll hX
\end{equation}
and
\begin{equation}
\label{eq:OutClaimE2}
\int_{X/2}^X \left| \sum_{\substack{z \leq p < 2X^{1/2}}} \left(1-\frac{\log p}{\log y}\right) \sum_u \lambda_{u, p}^+ E_{pu}(x)\right|^2 dx \ll hX.
\end{equation}
Proposition~\ref{prop:MeanSquare} below more-or-less reduces showing~\eqref{eq:OutClaimE1} to showing the following three claims:
\begin{align}
\label{eq:OutClaim1}
&\sum_{d \leq U} d \left( \sum_{\substack{u \leq U \\ u \equiv 0 \pmod{d}}} \frac{\lambda_{u}^-}{u}\right)^2 \ll \frac{1}{\log X},\\
\label{eq:OutClaim2}
&\sum_{0 < |k| \leq H} \left| \sum_{\substack{u_1, u_2 \leq U \\ (u_1, u_2) \mid k}} \lambda_{u_1}^- \lambda_{u_2}^- \left(\sum_{\substack{m_1 \\ X/2 < u_1 m_1 \leq X \\ u_1 m_1 \equiv  k \pmod{u_2}}} 1 - \frac{X/2}{[u_1, u_2]}\right)\right| \ll \frac{X}{\log X}, \\
\label{eq:OutClaim3}
& \sum_{n \in (X/2, X]} \left(\sum_{d \mid n}\lambda_d^-\right)^2 \ll \frac{X}{\log X}.
\end{align}
Here~\eqref{eq:OutClaim1} and~\eqref{eq:OutClaim3} follow from studying the structure of the sieve weights. Actually to make these estimates easier, we shall use the $\beta$-sieve with $\beta = 30$ for sieving small primes $p < X^\delta$ (see Section~\ref{se:set-up} for the choice of our sieve weights). 

Since $U$ is significantly larger than $X^{1/2}$, the claim~\eqref{eq:OutClaim2} is not obvious, but there is a well-known strategy for attacking it; for simplicity let us concentrate here on the case $(u_1, u_2) = 1$. We consider a weighted variant and use Poisson summation to the sum over $m_1 \equiv k \overline{u_1} \pmod{u_2}$ to relate it to averages of Kloosterman fractions of the type
\[
X \sum_{0 < |k| \leq H} \left|\sum_{\substack{u_1, u_2 \leq U \\ (u_1, u_2) = 1}} \frac{\lambda^-_{u_1} \lambda^-_{u_2}}{u_1 u_2} \sum_{0 < |\ell| \leq \frac{u_1 u_2}{X^{1-\varepsilon/20}}} e\left(\frac{k \ell \overline{u_1}}{u_2}\right) \right|.
\]
Such sums can be estimated using the work of Deshouillers and Iwaniec~\cite{DI} and its consequences. To apply these results, one needs some factorability properties of the coefficients $\lambda^-_{u_j}$. Here we can utilize the well-factorability of the linear sieve coefficients.

The claim~\eqref{eq:OutClaimE2} can be proved similarly, except in this case well-factorability is not so useful as $u$ is smaller. However, we can decompose the prime $p$ by Vaughan's identity and again finish by applying suitable bounds for averages of Kloosterman fractions. Also we will need to argue somewhat more carefully to avoid $\lambda^+_{p, u}$ depending on $p$, making it to depend only on a dyadic-type interval to which $p$ belongs.

In the above-mentioned work Mikawa~\cite{Mikawa} also used weighted sieve and estimates for Kloosterman sums but he did not take advantage of cancellations among the sieve weights for which reason he needed longer intervals (see Remark~\ref{rem:Greaves} below for more information about~\cite{Mikawa}).

\section{Setting up the sieves}
\label{se:set-up}
Let us introduce the set-up of Richert's~\cite{Richert} weighted sieve following~\cite[Chapter 25]{Opera}. For $x \in (X/2, X]$ and $2 \leq h \leq X^{1/100}$, write $\mathcal{A}(x) := (x-h\log X, x] \cap \mathbb{N}$ and, for any $z_0 \geq 2$, $P(z_0) := \prod_{p < z_0} p$. Define
\begin{equation}
\label{eq:defWzy}
D := X^{5/9}, \quad z := D^{1/4} = X^{5/36}, \quad y := D^{9/10} = X^{1/2},
\end{equation}
and
\[
w_n := 1-  \sum_{\substack{p \mid n \\ z \leq p < 2y}} \left(1-\frac{\log p}{\log y}\right).
\]
The coefficients $w_n$ have been chosen in such a way that we can prove that, for almost all $x \in (X/2, X]$,
\begin{equation}
\label{eq:wnaslowerbound}
\sum_{\substack{n \in \mathcal{A}(x) \\ (n, P(z)) = 1}} \mathbf{1}_{\Omega(n) \leq 2} \geq \frac{1}{2} \sum_{\substack{n \in \mathcal{A}(x) \\ (n, P(z)) = 1}} w_n
\end{equation}
and
\begin{equation}
\label{eq:wnhclaim}
\sum_{\substack{n \in \mathcal{A}(x) \\ (n, P(z)) = 1}} w_n \gg h.
\end{equation}
If we can show that these two claims hold for all $x \in (X/2, X]$ apart from an exceptional set of size $O(X/h)$, then Theorem~\ref{th:MT} clearly follows.

Let us first deduce~\eqref{eq:wnaslowerbound} which is much easier. We have, for $x \in (X/2, X]$,
\[
\begin{split}
\sum_{\substack{n \in \mathcal{A}(x) \\ (n, P(z)) = 1}} w_n &\leq \sum_{\substack{n \in \mathcal{A}(x) \\ (n, P(z)) = 1}} \left(1-  \sum_{\substack{p \mid n}} \left(1-\frac{\log p}{\log y}\right)\right) \leq \sum_{\substack{n \in \mathcal{A}(x) \\ (n, P(z)) = 1}} \left(1-  \left(\omega(n)-\frac{\log X}{\log y}\right)\right) \\
& = \sum_{\substack{n \in \mathcal{A}(x) \\ (n, P(z)) = 1}} \left(3 -  \omega(n) \right) \leq 2 \sum_{\substack{n \in \mathcal{A}(x) \\ (n, P(z)) = 1}} \mathbf{1}_{\omega(n) \leq 2}.
\end{split}
\]
There are only $\ll X/z$ integers $n \in (X/2, X]$ with $(n, P(z)) = 1$ for which $\Omega(n) > 2$ but $\omega(n) \leq 2$ (since such numbers are divisible by $p^2$ for some $p > z$). Hence, to deduce Theorem~\ref{th:MT} it indeed suffices to show that~\eqref{eq:wnhclaim} holds for all $x \in (X/2, X]$ apart from an exceptional set of measure $O(X/h)$.

Writing, for $\mathcal{B}\subseteq \mathbb{N}$,
\[
S(\mathcal{B}, z) := |\{n \in \mathcal{B}\colon (n, P(z)) =1\}| \quad \text{and} \quad \mathcal{B}_d = \{n \in \mathbb{N}\colon dn \in \mathcal{B}\},
\]
we have
\begin{equation}
\label{eq:wnSumSz}
\sum_{\substack{n \in \mathcal{A}(x) \\ (n, P(z)) = 1}} w_n = S(\mathcal{A}(x), z) -   \sum_{\substack{z \leq p < 2y}} \left(1-\frac{\log p}{\log y}\right) S(\mathcal{A}(x)_p, z).
\end{equation}

To find a lower bound for $S(\mathcal{A}(x), z)$, we introduce $\beta$-sieve and linear sieve weights (see e.g.~\cite[Section 6.4]{Opera}).
\begin{remark}
The reason that we do not use only the linear sieve is that using $\beta$-sieve (with e.g. $\beta = 30$) to sieve out primes $< X^\delta$ makes getting certain mean square estimates (like~\eqref{eq:alm-Claim} below) easier. However it is suggested in~\cite[between Corollary 6.28 and Proposition 6.29]{Opera} that one could prove such mean square estimates also for the linear sieve alone.

On the other hand, the reason that we do not use only the $\beta$-sieve with $\beta = 30$ is that the linear sieve leads to superior sieving results --- in particular our lower bound for~\eqref{eq:wnSumSz} would be negative if we only used the $\beta$-sieve.
\end{remark}

Let $\beta = 30$, let $\delta > 0$ be small and take $w = X^{\delta}$ and $E = X^{1/1000}$. Write also $P(w, z) = \prod_{w \leq p < z} p$ and define
\[
\begin{split}
\mathcal{D}^+ &:= \{d = p_1 \dotsm p_r \mid P(w, z) \colon p_1 > p_2 > \dotsc > p_r, p_1 \dotsm p_m p_m^2 < D \text{ for all odd $m$} \}, \\
\mathcal{D}^- &:= \{d = p_1 \dotsm p_r \mid P(w, z) \colon p_1 > p_2 > \dotsc > p_r, p_1 \dotsm p_m p_m^2 < D \text{ for all even $m$} \}, \\
\mathcal{E}^+ &:= \{e = p_1 \dotsm p_r \mid P(w) \colon p_1 > p_2 > \dotsc > p_r, p_1 \dotsm p_m p_m^\beta < E \text{ for all odd $m$} \}, \\
\mathcal{E}^- &:= \{e = p_1 \dotsm p_r \mid P(w) \colon p_1 > p_2 > \dotsc > p_r, p_1 \dotsm p_m p_m^\beta < E \text{ for all even $m$} \}.
\end{split}
\]
Now define the upper and lower bound linear sieve weights $\lambda_d^\pm = \mu(d) \mathbf{1}_{d \in \mathcal{D}^\pm}$ and the upper and lower bound $\beta$-sieve weights $\rho_e^\pm = \mu(e) \mathbf{1}_{e \in \mathcal{E}^\pm}$, so that, for any $n \in \mathbb{N}$, (see e.g.~\cite[Equations (6.26) and (6.27) with $\mathcal{A} = \{n\}$]{Opera})
\begin{equation}
\label{eq:sieves}
\begin{split}
\sum_{d \mid n} \lambda_d^- &\leq \mathbf{1}_{(n, P(w, z)) = 1} \leq \sum_{d \mid n} \lambda_d^+, \\
\sum_{e \mid n} \rho_e^- &\leq \mathbf{1}_{(n, P(w)) = 1} \leq \sum_{e \mid n} \rho_e^+.
\end{split}
\end{equation}

We cannot obtain a lower bound for $\mathbf{1}_{(n, P(z))}$ directly by multiplying the lower bounds for $\mathbf{1}_{(n, P(w, z)) = 1}$ and $\mathbf{1}_{(n, P(w)) = 1}$ since for some $n$ both lower bounds might be negative. However, we can use~\eqref{eq:sieves} to derive a lower bound for $\mathbf{1}_{(n, P(z)) = 1}$ that is familiar from the vector sieve (see e.g.~\cite[Lemma 10.1]{HarmanBook}):
\[
\begin{split}
\mathbf{1}_{(n, P(z)) = 1} &= \mathbf{1}_{(n, P(w, z)) = 1} \mathbf{1}_{(n, P(w)) = 1} \\
&= \left(\sum_{d \mid n} \lambda_d^+\right) \mathbf{1}_{(n, P(w)) = 1} - \left(\sum_{d \mid n} \lambda_d^+ - \mathbf{1}_{(n, P(w, z)) = 1}\right) \mathbf{1}_{(n, P(w)) = 1} \\
&\geq \sum_{d \mid n} \lambda_d^+ \sum_{e \mid n} \rho_e^- - \left(\sum_{d \mid n} \lambda_d^+ - \mathbf{1}_{(n, P(w, z)) = 1}\right) \sum_{e \mid n} \rho_e^+ \\
&\geq \sum_{d \mid n} \sum_{e \mid n} \left(\lambda_d^+ \rho_e^- - \lambda_d^+ \rho_e^+ + \lambda_d^- \rho_e^+ \right) = \sum_{k \mid n} \alpha_k^-,
\end{split}
\]
where
\begin{equation}
\label{eq:alpk-def}
\alpha_k^- = \mathbf{1}_{k \mid P(z)} \left( \lambda_{(k, P(w, z))}^+ \rho_{(k, P(w))}^- + \lambda_{(k, P(w, z))}^- \rho_{(k, P(w))}^+ - \lambda_{(k, P(w, z))}^+ \rho_{(k, P(w))}^+\right)
\end{equation}
say. Hence
\begin{equation}
\label{eq:nPzlow}
S(\mathcal{A}(x), z) \geq \sum_{d \mid P(z)} \alpha_d^- |\mathcal{A}(x)_d|.
\end{equation}
Note that $\alpha_k^{-}$ are supported on $k \leq DE$, so they are lower bound sieve weights with level $DE$.

Let us now turn to obtaining an upper bound for $S(\mathcal{A}(x)_p, z)$. If we can obtain level of distribution $DE$ for $\mathcal{A}(x)$, we can typically apply a sieve of level $DE/p$ to $\mathcal{A}(x)_p$. However, it will be technically convenient if the level is more stable when $p$ varies and if $p$ has a smooth weight.

To achieve this we introduce a smooth partition of the unity. Let $\psi \colon \mathbb{R}_+ \to [0, 1]$ be a smooth function such that $\psi(x) = 0$ for $x \leq 1$, $\psi(x) = 1$ for $x \geq \sqrt{2}$, and
\[
\frac{d^k}{dx^k} \psi(x) \ll_k 1 \quad \quad \text{for every $k \in \mathbb{N}$.}
\]
Defining then $\sigma \colon \mathbb{R}_{+} \to [0, 1]$ by
\begin{equation}
\label{eq:sigmadef}
\sigma(x) := 
\begin{cases}
\psi(x) & \text{if $0 < x \leq \sqrt{2}$;} \\
1-\psi\left(\frac{x}{\sqrt{2}}\right) & \text{if $x > \sqrt{2}$,}
\end{cases}
\end{equation}
the function $\sigma(x)$ is compactly supported in $[1, 2]$, and for all $x \in \mathbb{R}_+$ we have
\begin{equation}
\label{eq:sigmapart}
\sum_{a \in \mathbb{Z}} \sigma\left(\frac{x}{\sqrt{2}^a}\right) = 1,
\end{equation}
and
\begin{equation}
\label{eq:sigmaderbound}
\frac{d^k}{dx^k} \sigma(x) \ll_k 1 \quad \quad \text{for every $k \in \mathbb{N}$.}
\end{equation}
Consequently, writing 
\begin{equation}
\label{eq:Idef}
\mathcal{I} = \left[\left\lfloor \frac{\log z}{\log \sqrt{2}}\right\rfloor - 2, \left\lfloor \frac{\log y}{\log \sqrt{2}}\right\rfloor \right] \cap \mathbb{N},
\end{equation}
we have, for any $p \in \mathbb{P}$,
\begin{equation}
\label{eq:sumSigmaoverI} 
\sum_{a \in \mathcal{I}} \sigma\left(\frac{p}{\sqrt{2}^a}\right) \quad
\begin{cases}
= 1, & \text{if $z \leq p < y$;} \\
= 0, & \text{if $p < z/4$ or $p > 2y$;} \\
\in [0, 1], & \text{otherwise.}
\end{cases}
\end{equation}
Hence
\begin{equation}
\label{eq:AddSmoothing}
\sum_{\substack{z \leq p < 2y}} \left(1-\frac{\log p}{\log y}\right) S(\mathcal{A}(x)_p, z) \leq \sum_{a \in \mathcal{I}} \sum_p \sigma\left(\frac{p}{\sqrt{2}^a} \right)\left(1-\frac{\log p}{\log y}\right) S(\mathcal{A}(x)_p, z).
\end{equation}
Note that, for $a \in \mathcal{I}$, the smooth weight $\sigma(p/\sqrt{2}^a)$ is supported on 
\[
p \in [\sqrt{2}^a, \sqrt{2}^{a+2}] \subseteq [z/4, 2y].
\]

For $a \in \mathcal{I}$, let
\begin{equation}
\label{eq:Dadef}
D_a = D/\sqrt{2}^{a+2}
\end{equation}
and
\[
\begin{split}
\mathcal{D}_a^+ &:= \{d = p_1 \dotsm p_r \mid P(w, z) \colon p_1 > p_2 > \dotsc > p_r, p_1 \dotsm p_m p_m^2 < D_a \text{ for all odd $m$} \}.
\end{split}
\]
Define the upper bound linear sieve weights $\lambda^+_{d, a} = \mu(d) \mathbf{1}_{d \in \mathcal{D}^+_{a}}$, so that, for any $a \in \mathcal{I}$ and $n \in \mathbb{N}$,
\[
\mathbf{1}_{(n, P(w, z)) = 1} \leq \sum_{d \mid n} \lambda^+_{d, a}.
\] 
Recalling also~\eqref{eq:sieves} we see that, for any $a \in \mathcal{I}$ and $n \in \mathbb{N}$,
\begin{equation}
\label{eq:nP(z)upper}
\mathbf{1}_{(n, P(z)) = 1} = \mathbf{1}_{(n, P(w, z)) = 1} \mathbf{1}_{(n, P(w)) = 1} \leq \left(\sum_{\substack{d \mid n}}\lambda^+_{d, a} \right) \left(\sum_{e \mid n} \rho_e^+\right) = \sum_{k \mid n} \alpha_{k, a}^+,
\end{equation}
where
\begin{equation}
\label{eq:alphaka+def}
\alpha_{k,a}^+ := \mathbf{1}_{k \mid P(z)} \lambda_{(k, P(w, z)), a}^+ \rho_{(k, P(w))}^+.
\end{equation}
Note that $\alpha_{k, a}^+$ are supported on $k \leq D_a E = DE/\sqrt{2}^{a+2}$.

Combining~\eqref{eq:wnSumSz} and~\eqref{eq:AddSmoothing} and then using~\eqref{eq:nPzlow} and~\eqref{eq:nP(z)upper} we obtain 
\[
\begin{split}
\sum_{\substack{n \in \mathcal{A}(x) \\ (n, P(z)) = 1}} w_n &\geq S(\mathcal{A}(x), z) -  \sum_{a \in \mathcal{I}} \sum_p \sigma\left(\frac{p}{\sqrt{2}^a} \right)\left(1-\frac{\log p}{\log y}\right)  \sum_{n \in \mathcal{A}(x)_p} \mathbf{1}_{(n, P(z)) = 1} \\
&\geq \sum_{d \mid P(z)} \alpha_d^- |\mathcal{A}(x)_d| -  \sum_{a \in \mathcal{I}} \sum_p \sigma\left(\frac{p}{\sqrt{2}^a} \right)\left(1-\frac{\log p}{\log y}\right)  \sum_{d \mid P(z)} \alpha_{d, a}^+ |\mathcal{A}(x)_{dp}|.
\end{split}
\]
Writing, for $e \in \{d, dp\}$,
\[
|\mathcal{A}(x)_e| = \frac{h\log X}{e} + \left( |\mathcal{A}(x)_e| - \frac{h\log X}{e}\right)
\]
we see that, for every $x \in (X/2, X]$,
\[
\sum_{\substack{n \in \mathcal{A}(x) \\ (n, P(z)) = 1}} w_n \geq h \log X \cdot M(z, y) + E^-(x, y, z) - E^+(x, y, z),
\]
where
\[
\begin{split}
M(z, y) &:= \sum_{d \mid P(z)} \frac{\alpha_d^-}{d} -  \sum_{a \in \mathcal{I}} \sum_p \sigma\left(\frac{p}{\sqrt{2}^a} \right) \left(1-\frac{\log p}{\log y}\right) \sum_{d \mid P(z)} \frac{\alpha_{d, a}^+}{dp}, \\
E^-(x, y, z) &:= \sum_{d \mid P(z)} \alpha_d^- \left(|\mathcal{A}(x)_d|- \frac{h\log X}{d}\right), \\
E^+(x, y, z) &:=\sum_{a \in \mathcal{I}} \sum_p \sigma\left(\frac{p}{\sqrt{2}^a} \right) \left(1-\frac{\log p}{\log y}\right) \sum_{d \mid P(z)} \alpha_{d, a}^+ \left(|\mathcal{A}(x)_{dp}|-\frac{h\log X}{dp}\right).
\end{split}
\]

Hence, in order to establish that~\eqref{eq:wnhclaim} holds for all $x\in (X/2, X]$ apart from an exceptional set of measure $O(X/h)$, it suffices to show that
\begin{equation}
\label{eq:M(zy)claim}
M(z, y) \gg \frac{1}{\log X}
\end{equation}
and that
\begin{equation}
\label{eq:E(x)claim}
\int_{X/2}^{X} |E^\pm(x, y, z)|^2 dx \ll h X.
\end{equation}
We will establish~\eqref{eq:M(zy)claim} in Section~\ref{se:MainTermM}. In Section~\ref{se:Lemmas} we collect some lemmas needed in establishing~\eqref{eq:E(x)claim}. Then we will do some preliminary work on type I sums in almost all very short intervals in Section~\ref{se:TypeI} before establishing~\eqref{eq:E(x)claim} in Section~\ref{se:E(x)}.

\begin{remark}
\label{rem:Greaves}
We have not optimized the level of distribution or the sieve weights as the current set-up suffices for obtaining $P_2$-numbers. As pointed out to the author by James Maynard and Maksym Radziwi{\l\l}, it might be possible to alternatively use Greaves' most sophisticated weighted sieve~\cite{Greaves} together with Bettin-Chandee~\cite{BeCha} estimates for Kloosterman sums. In this alternative approach the estimation of $S_2$ from Proposition~\ref{prop:MeanSquare} below would be simpler whereas the sieve weights and thereby the estimation of $S_1$ would become more complicated.

On the other hand, after the completion of this work, the author realised, thanks to a comment by Andrew Granville, that it would probably suffice to use Kloosterman sum estimates based on the Weil bound as Mikawa~\cite{Mikawa} does. This would again simplify the treatment of $S_2$. However, our results in Section~\ref{se:TypeI} give better bilinear level of distribution in almost all short intervals, which might be of benefit for other applications, so we have decided to keep the current approach.
\end{remark}

\section{Handling the main term $M(z, y)$}
\label{se:MainTermM}
Take a small $\varepsilon' > 0$ and, for $z_0 \geq 2$, write $V(z_0) := \prod_{p < z_0} (1-1/p)$. Recall that $w = X^\delta$. By the fundamental lemma of the sieve (see e.g.~\cite[(6.31)--(6.33) and Lemma 6.8]{Opera}), we have, once $\delta$ is small enough in terms of $\varepsilon'$,
\begin{equation}
\label{eq:FLsieve}
(1-\varepsilon') V(w) \leq \sum_{e \mid P(w)} \frac{\rho_e^-}{e} \leq V(w) \leq \sum_{e \mid P(w)} \frac{\rho_e^+}{e} \leq (1+\varepsilon') V(w) 
\end{equation}

Let $F(s)$ and $f(s)$ be the linear sieve functions (see e.g. in~\cite[(12.1, 12.2)]{Opera}) so that in particular 
\begin{equation}
\label{eq:f(4)F(s)}
f(4) = e^\gamma (\log 3)/2 \quad \text{and, for $0 < s \leq 3$}, F(s) = 2e^\gamma/s.
\end{equation} 
By the linear sieve theory (see e.g.~\cite[(12.4, 12.5)]{Opera} --- note that~\cite[(12.4)]{Opera} actually holds for $s > 0$) we have
\begin{equation}
\label{eq:LinSieve}
\begin{split}
\sum_{d \mid P(w.z)} \frac{\lambda_{d,a}^+}{d} &\leq \left(F\left(\frac{\log D_a}{\log z}\right)+\varepsilon'\right) \prod_{w \leq p < z} \left(1-\frac{1}{p}\right), \\
\sum_{d \mid P(w.z)} \frac{\lambda_{d}^-}{d} &\geq \left(f\left(\frac{\log D}{\log z}\right)-\varepsilon'\right) \prod_{w \leq p < z} \left(1-\frac{1}{p}\right), \\
\sum_{d \mid P(w.z)} \frac{\lambda_{d}^+}{d} &\leq \left(F\left(\frac{\log D}{\log z}\right)+\varepsilon'\right) \prod_{w \leq p < z} \left(1-\frac{1}{p}\right).
\end{split}
\end{equation}
Recall that $z = D^{1/4}$. Hence, once $\delta$ is small enough in terms of $\varepsilon'$, the definition of $M(z, y)$ and the sieve bounds~\eqref{eq:FLsieve} and~\eqref{eq:LinSieve} imply that
\[
\begin{split}
\frac{M(z, y)}{V(z)} &= \frac{1}{V(z)} \Biggl( \sum_{d \mid P(w,z)} \frac{\lambda_d^-}{d} \sum_{e \mid P(w)} \frac{\rho_e^+}{e} + \sum_{d \mid P(w,z)} \frac{\lambda_d^+}{d} \left(\sum_{e \mid P(w)} \frac{\rho_e^-}{e}  - \sum_{e \mid P(w)} \frac{\rho_e^+}{e} \right)\\
& \qquad -  \sum_{a \in \mathcal{I}} \sum_p \sigma\left(\frac{p}{\sqrt{2}^a} \right) \left(1-\frac{\log p}{\log y}\right) \frac{1}{p} \sum_{d \mid P(w, z)} \frac{\lambda_{d, a}^+}{d}\sum_{e \mid P(w)} \frac{\rho_e^+}{e} \Biggr) \\
&\geq (f(4)- \varepsilon') (1-\varepsilon') - 2\varepsilon' \frac{V(w)}{V(z)}\sum_{d \mid P(w,z)} \frac{\lambda_d^+}{d} \\
& \quad -  \sum_{a \in \mathcal{I}} \sum_p \sigma\left(\frac{p}{\sqrt{2}^a} \right) \left(1-\frac{\log p}{\log y}\right) \frac{1}{p}  \left(F\left(\frac{\log D_a}{\log z}\right) + \varepsilon' \right) (1+\varepsilon') \\
&\geq f(4) -  \sum_{z < p \leq y} \left(1-\frac{\log p}{\log y}\right) \frac{1}{p} F\left(\frac{\log D/p}{\log z}\right) -100 \varepsilon'.
\end{split}
\]
Now we are in the situation of~\cite[Section 25.3 with $s = 4, u = 10/9,$ and $\eta = 1$]{Opera} but for completeness we evaluate the lower bound also here.

Plugging in the values from~\eqref{eq:f(4)F(s)} and evaluating the sums over $p$ by the prime number theorem and then substituting $t = D^\alpha$, we get
\[
\begin{split}
\frac{M(z, y)}{V(z)} &\geq \frac{e^\gamma \log 3}{2} - 2 e^\gamma \int_{z}^y \left(1-\frac{\log t}{\log y}\right) \frac{1}{t} \frac{\log D^{1/4}}{\log (D/t)} \frac{dt}{\log t} -200 \varepsilon' \\
&= \frac{e^\gamma \log 3}{2} - 2 e^\gamma \int_{1/4}^{9/10} \left(1-\frac{10\alpha}{9}\right) \frac{1}{4(1-\alpha)} \frac{d\alpha}{\alpha} -200 \varepsilon'.
\end{split}
\]
Evaluating the integral, we obtain
\[
\frac{M(z, y)}{V(z)} \geq \frac{e^{\gamma}}{2}\log 3 \left(1 - \frac{1}{\log 3}\left(\log(27) - \frac{10}{9} \log(15/2) \right)\right) - 200\varepsilon'.
\]
By a numerical calculation and~\eqref{eq:Mertens} below we see that indeed $M(z, y) \gg 1/\log X$ once $\varepsilon'$ is small enough.

\section{Auxiliary results}
\label{se:Lemmas}
Before turning to proving~\eqref{eq:E(x)claim} we collect here some known auxiliary results. We will use some standard estimates for multiplicative functions. Note first that, for any divisor-bounded (i.e. a function bounded by $d(n)^C$ for some $C$) multiplicative function $f \colon \mathbb{N} \to \mathbb{C}$, we have
\begin{equation}
\label{eq:f(n)/n_average}
\sum_{n \leq X} \frac{|f(n)|}{n} \ll \prod_{p \leq X} \left(1+\frac{|f(p)|}{p}\right).
\end{equation}
Furthermore, for $k \in \mathbb{R}$ and $z \geq w \geq 2$,
\begin{equation}
\label{eq:Mertens}
\prod_{w < p \leq z} \left(1+\frac{k}{p}\right) \asymp_k \left(\frac{\log z}{\log w}\right)^k.
\end{equation}
The following consequence of Shiu's~\cite{Shiu} bound allows us to estimate divisor sums.
\begin{lemma}
\label{le:Shiu}
Let $m \geq 1$ and let $X \geq z \geq 2$. Then
\[
\sum_{n \leq X} \tau(n)^m \mathbf{1}_{(n, P(z)) = 1} \ll_m \frac{X}{\log X} \cdot \left(\frac{\log X}{\log z}\right)^{2^m}.
\]
\end{lemma}
\begin{proof}
By Shiu's bound~\cite[Theorem 1]{Shiu}
\[
\begin{split}
\sum_{n \leq X} \tau(n)^m \mathbf{1}_{(n, P(z)) = 1} &\ll_m X \prod_{p \leq X} \left(1+\frac{2^m \mathbf{1}_{p > z}  - 1}{p}\right)
\end{split}
\]
and the claim follows immediately from~\eqref{eq:Mertens}.
\end{proof}

Next we record Vaughan's identity in a form that is convenient for us.
\begin{lemma}\label{le:Vaughan-identity}
Let $X \geq 2$. There exists $k \ll (\log X)^2$ such that, for each $X < n \leq 2X$, one has
\[
\Lambda(n) = \sum_{j=1}^k \sum_{n = uv} a_j(u) b_j(v),
\]
where $a_j(u)$ and $b_j(v)$ are real coefficients such that
\begin{enumerate}[(i)]
\item For each $j = 1, \dotsc, k$ and every $n$, one has $|a_j(n)|, |b_j(n)| \leq 1+\log n$.
\item For each $j = 1, \dotsc, k$, there exist $U_j \in (1/2, X^{1/2}]$ and $V_j \in [X^{1/2}/2, 2X]$ such that $a_j(u)$ are supported on $u \in (U_j, 2U_j]$ and $b_j(v)$ are supported on $v \in (V_j, 2V_j]$. Moreover $U_j V_j \in (X/4, 2X]$.
\item For each $j$ with $V_j \geq 4X^{2/3}$ one has $b_j(v) = \sigma(v/V_j) \log v$ or $b_j(v) = \sigma(v/V_j)$ where $\sigma(x)$ is as in~\eqref{eq:sigmadef}.
\end{enumerate}
\end{lemma}
\begin{proof}
By \cite[Proposition 13.4]{IwaKov} with $y = z = X^{1/3}$, we have, for $n \in (X, 2X]$,
\begin{equation}
\label{eq:VaugOrg}
\Lambda(n) = \sum_{\substack{n = ab \\ b \leq X^{1/3}}} \mu(b) \log a - \sum_{\substack{n = abc \\ b, c \leq X^{1/3}}} \mu(b) \Lambda(c) + \sum_{\substack{n = abc \\ b, c > X^{1/3}}} \mu(b) \Lambda(c) =:S_1(n) - S_2(n) + S_3(n),
\end{equation}
say.
Let us show that $S_2(n)$ can be written as a sum of $O((\log X)^2)$ sums of the form $\sum_{n = uv} a(u) b(v)$ with $a(u), b(v), U,$ and $V$ as $a_j(u), b_j(v), U_j,$ and $V_j$ above. One can deal with $S_1(n)$ and $S_3(n)$ similarly.

Consider $n \in (X, 2X]$. In $S_2(n)$ we write $bc = k$ and note that $k \in [1, X^{2/3}]$ and $a \in [X^{1/3}, 2X]$. We split the variable $k$ into dyadic ranges and make a smooth dyadic partition of the variable $a$ recalling~\eqref{eq:sigmapart}. We obtain, for $X < n \leq 2X$, 
\[
S_2(n) = \sum_{\substack{i, j \\ X^{1/3}/2 \leq \sqrt{2}^i \leq 2X \\ \frac{1}{2} \leq 2^j \leq X^{2/3}}} \sum_{\substack{n = ak}} \sigma\left(\frac{a}{\sqrt{2}^i}\right) \left(\mathbf{1}_{k \sim 2^j} \sum_{\substack{k = bc \\ b, c \leq X^{1/3}}} \mu(b) \Lambda(c)\right) .
\]
The first sum runs over $O((\log X)^2)$ pairs $(i, j)$. For each such pair, the sum over $n = ak$ is of the desired shape, with $U = \min\{2^j, \sqrt{2}^i\}$ and $V = \max\{2^j, \sqrt{2}^i\}$. In particular the requirement (i) holds since $\sum_{c \mid k} \Lambda(c) = \log k$.
\end{proof}

The following lemma gives two convenient consequences of the Poisson summation formula.
\begin{lemma}
\label{le:Poisson}
Let $f \colon \mathbb{R} \to \mathbb{R}$ be such that $f$ and $\widehat{f}$ are in $L^1(\mathbb{R})$ and have bounded variation.
\begin{enumerate}[(i)]
\item For any $u \in \mathbb{R}$ and $v \in \mathbb{R}^+$, one has
\[
\sum_{\substack{n \in \mathbb{Z}}} f\left(vn + u\right) = \frac{1}{v} \sum_{h \in \mathbb{Z}} \widehat{f}\left(\frac{h}{v}\right)e\left(\frac{uh}{v}\right). 
\]
\item
For any $a, q \in \mathbb{N}$ and $Y > 0$, one has
\[
\sum_{\substack{n \in \mathbb{Z} \\ n \equiv a \pmod{q}}} f\left(\frac{n}{Y}\right) = \frac{Y}{q} \sum_{h \in \mathbb{Z}} \widehat{f}\left(\frac{Y}{q} h \right)e\left(\frac{ah}{q}\right). 
\]
\end{enumerate}
\end{lemma}
\begin{proof}
Part (i) is~\cite[Formula (4.24)]{IwaKov}. Part (ii) follows from part (i) after writing 
\[
\sum_{\substack{n \in \mathbb{Z} \\ n \equiv a \pmod{q}}} f\left(\frac{n}{Y}\right)  = \sum_{m \in \mathbb{Z}} f\left(\frac{q}{Y}m + \frac{a}{Y}\right).
\]
\end{proof}

Let us finally record two lemmas that we use to bound averages of Kloosterman fractions. The first one is~\cite[Lemma 1]{DIzeta2} with $\varrho = 1$.
\begin{lemma}
\label{le:KloTypeII}
Let $C, D, U, V \geq 1$ and $|c(u, v)| \leq 1$. Then, for any $\varepsilon > 0$,
\[
\begin{split}
&\sum_{1 \leq c \leq C} \sum_{\substack{1 \leq d \leq D \\ (c, d) = 1}} \left|\sum_{1 \leq u \leq U} \sum_{\substack{1 \leq v \leq V \\ (v, c) = 1}} c(u, v) e\left(u \frac{\overline{vd}}{c}\right)\right| \\
&\ll (CDUV)^{1/2+\varepsilon}\left( (CD)^{1/2} + (U+V)^{1/4}\left(CD(U+V)(C+V^2) + UV^2D^2\right)^{1/4}\right).
\end{split}
\]
\end{lemma}

The second one is an immediate consequence of~\cite[Theorem 12]{DI}.
\begin{lemma}
\label{le:KloTypeI}
Let $C, D, N, R, S \geq 1/2$ and let $b_{n, r, s}$ be bounded complex coefficients. Let $g \colon \mathbb{R}^2 \to \mathbb{R}$ be a smooth compactly supported function such that
\begin{equation}
\label{eq:g2dimderbound}.
\left|\frac{\partial^{\nu_1 + \nu_2}}{\partial x_1^{\nu_1} \partial x_2^{\nu_2}} g(x_1, x_2)\right| \ll_{\nu_1, \nu_2} 1 \quad \quad \text{for every $\nu_1, \nu_2 \geq 0.$}
\end{equation}
Then, for any $\varepsilon > 0$,
\[
\begin{split}
&\sum_{\substack{R < r \leq 2R \\ S < s \leq 2S \\ (r, s) = 1}} \sum_{0 < n \leq N} b_{n, r, s}\sum_{\substack{c, d \\ (rd, sc) = 1}} g\left(\frac{c}{C}, \frac{d}{D}\right) e\left(n \frac{\overline{rd}}{sc}\right) \\
&\ll (CD)^\varepsilon (NRS)^{1/2+\varepsilon}\left(CS(RS+N)(C+DR)+C^2DS\sqrt{(RS+N)R} + D^2NR/S\right)^{1/2}.
\end{split}
\]
\end{lemma}

\section{Type I sums in almost all short intervals}
\label{se:TypeI}
We shall use the following general result as a starting point for showing~\eqref{eq:E(x)claim}.
\begin{proposition}
\label{prop:MeanSquare}
Let $X \geq H \geq 2$. Let $g \colon \mathbb{R} \to [0, 1]$ be a smooth function that is compactly supported on $[1/4, 2]$ and satisfies~\eqref{eq:gderivbound}. For $d \in \mathbb{N}$, define
\[
\gamma_{d, H} := \sum_{\substack{m \geq 1 \\ (m, d) = 1}} \left(\frac{d}{\pi m} \sin\left(\frac{\pi m H}{d}\right)\right)^2.
\]
Let $a_d \in \mathbb{R}$ be bounded for all $d$. Let $2 \leq D_0 \leq X^{1-\delta}$ for some fixed $\delta \in (0, 1)$. Then
\[
\int_{-\infty}^{\infty} g\left(\frac{x}{X}\right) \left(\sum_{\substack{d \leq D_0, m \in \mathbb{N} \\ x-H < dm \leq x}} a_d - H \sum_{d \leq D_0} \frac{a_d}{d}\right)^2 dx = S_1 + S_2 + S_3 + O(H^3 (\log X)^3),
\]
where
\[
\begin{split}
S_1 &:= 2\widehat{g}(0) X \sum_{d \leq D_0} \gamma_{d, H} \left( \sum_{\substack{m \leq D_0 \\ m \equiv 0 \pmod{d}}} \frac{a_{m}}{m}\right)^2,\\
S_2 &:= \sum_{0 < |k| \leq H} (H- |k|) \sum_{\substack{d_1, d_2 \leq D_0 \\ (d_1, d_2) \mid k}} a_{d_1} a_{d_2} \left(\sum_{\substack{m_1, m_2 \\ d_1 m_1 = d_2 m_2 + k}} g\left(\frac{d_1 m_1}{X}\right) - \frac{\widehat{g}(0) X}{[d_1, d_2]}\right), \\
S_3 &:= H \sum_{n} g\left(\frac{n}{X}\right) \left(\sum_{d \mid n}a_d\right)^2 - \widehat{g}(0) H X  \frac{1}{X^{10}} \sum_{n \leq X^{10}}  \left(\sum_{d \mid n}a_d\right)^2.
\end{split}
\]
\end{proposition}
\begin{remark}
This can be compared with~\cite[Proposition 6.25]{Opera} which is non-trivial for $D_0 < X^{1/2}(\log X)^{-C}$. For our choices of $a_d$ we will be able to estimate $S_j$ succesfully for a wider range of $D_0$.
\end{remark}

\begin{proof}[Proof of Proposition~\ref{prop:MeanSquare}]
We start by squaring out, obtaining
\begin{equation}
\label{eq:Ssquaredout}
\begin{split}
S &:= \int_{-\infty}^\infty g\left(\frac{x}{X}\right) \left(\sum_{\substack{d \leq D_0, m \in \mathbb{N} \\ x-H < dm \leq x}} a_d - H \sum_{d \leq D_0} \frac{a_d}{d}\right)^2 dx \\
&=  \int_{-\infty}^\infty g\left(\frac{x}{X}\right)  \left(\sum_{\substack{d \leq D_0, m \in \mathbb{N} \\ x-H < d m \leq x}} a_d\right)^2 dx - 2 H \sum_{d_1 \leq D_0} \frac{a_{d_1}}{d_1}  \int_{-\infty}^\infty g\left(\frac{x}{X}\right)  \left(\sum_{\substack{d_2 \leq D_0, m \in \mathbb{N} \\ x-H < d_2 m \leq x}} a_{d_2}\right) dx \\
& \qquad \qquad + H^2 \int_{-\infty}^\infty g\left(\frac{x}{X}\right) dx \left(\sum_{d \leq D_0} \frac{a_d}{d}\right)^2.
\end{split}
\end{equation}
Here
\begin{equation}
\label{eq:midtermtypeIProp}
\begin{split}
&\int_{-\infty}^\infty g\left(\frac{x}{X}\right)  \left(\sum_{\substack{d_2 \leq D_0, m \in \mathbb{N} \\ x-H < d_2 m \leq x}} a_{d_2}\right) dx = \sum_{d_2 \leq D_0} a_{d_2} \sum_m \int_{d_2 m}^{d_2 m + H}  g\left(\frac{x}{X}\right) dx \\
& = \sum_{d_2 \leq D_0} a_{d_2} \int_0^H \sum_m g\left(\frac{d_2 m + t}{X}\right) dt.
\end{split}
\end{equation}
Applying the Poisson summation (Lemma~\ref{le:Poisson}(i)) and~\eqref{eq:gFourbound} we get, for every $d_2 \leq D_0 \leq X^{1-\delta}$,
\[
\begin{split}
\sum_m g\left(\frac{d_2 m + t}{X}\right) &= \frac{X}{d_2} \sum_{h \in \mathbb{Z}} \widehat{g}\left(\frac{hX}{d_2}\right)e\left(\frac{th}{d_2}\right) = \frac{X}{d_2} \widehat{g}(0) + O\left(\frac{X}{d_2} \sum_{\substack{h \in \mathbb{Z} \\ h \neq 0}} \left(\frac{d_2}{hX}\right)^{\frac{10}{\delta}}\right) \\
&= \frac{X}{d_2} \widehat{g}(0) + O(X^{-9}).
\end{split}
\]
Using this in~\eqref{eq:midtermtypeIProp} we see that
\[
\int_{-\infty}^\infty g\left(\frac{x}{X}\right)  \left(\sum_{\substack{d_2 \leq D_0, m \in \mathbb{N} \\ x-H < d_2 m \leq x}} a_{d_2}\right) dx = HX \widehat{g}(0) \sum_{d_2 \leq D_0} \frac{a_{d_2}}{d_2} + O(HD_0 X^{-9}).
\]
Substituting this into~\eqref{eq:Ssquaredout}, we obtain
\[
\begin{split}
S &= \int_{-\infty}^{\infty} g\left(\frac{x}{X}\right) \left(\sum_{\substack{d \leq D_0, m \in \mathbb{N} \\ x-H < dm \leq x}} a_d\right)^2 dx - H^2 X \widehat{g}(0) \left(\sum_{d \leq D_0} \frac{a_d}{d}\right)^2 + O\left(\frac{H^2 D_0 \log X}{X^{9}}\right).
\end{split}
\]
Squaring out, the first term equals
\[
\begin{split}
&\sum_{\substack{d_1, d_2 \leq D_0 \\ m_1, m_2 \\ |d_1 m_1 - d_2 m_2| \leq H}} a_{d_1} a_{d_2} \int_{-\infty}^{\infty} g\left(\frac{x}{X}\right)\mathbf{1}_{x-H < d_1m_1, d_2m_2 \leq x} dx  \\
&= \sum_{|k| \leq H}   \sum_{\substack{d_1, d_2 \leq D_0 \\ m_1, m_2 \\ d_1 m_1 = d_2 m_2 + k}}  a_{d_1} a_{d_2} \cdot (H- |k|)  \left(g\left(\frac{d_1 m_1}{X}\right) + O\left(\frac{H}{X}\right)\right)  \\
&= \sum_{|k| \leq H} (H- |k|) \sum_{\substack{d_1, d_2 \leq D_0 \\ (d_1, d_2) \mid k}}  a_{d_1} a_{d_2} \sum_{\substack{m_1, m_2 \\ d_1 m_1 = d_2 m_2 + k}} g\left(\frac{d_1 m_1}{X}\right) +  O\left(\frac{H^2}{X}\sum_{|k| \leq H} \sum_{X/4 \leq n \leq 2X} \tau(n) \tau(n+k)\right).
\end{split}
\]
The error term here is by the inequality $|xy| \leq x^2 + y^2$ and the Shiu bound (Lemma~\ref{le:Shiu}) 
\[
\ll \frac{H^2}{X}\sum_{|k| \leq H} \sum_{X/4 \leq n \leq 2X} (\tau(n)^2 + \tau(n+k)^2) \ll \frac{H^3}{X} \sum_{X/5 \leq n \leq 3X} \tau(n)^2 \ll H^3 (\log X)^3.
\]
Consequently, subtracting and adding the expected main term,
\begin{equation}
\label{eq:Smiddle}
\begin{split}
S &= \sum_{|k| \leq H} (H- |k|) \sum_{\substack{d_1, d_2 \leq D_0 \\ (d_1, d_2) \mid k}}  a_{d_1} a_{d_2} \left(\sum_{\substack{m_1, m_2 \\ d_1 m_1 = d_2 m_2 + k}} g\left(\frac{d_1 m_1}{X}\right) - \frac{\widehat{g}(0) X}{[d_1, d_2]}\right) \\
&+ \widehat{g}(0) X \sum_{|k| \leq H} (H- |k|) \sum_{\substack{d_1, d_2 \leq D_0 \\ (d_1, d_2) \mid k}} \frac{a_{d_1} a_{d_2}}{[d_1, d_2]} - H^2 X \widehat{g}(0) \left(\sum_{d \leq D_0} \frac{a_d}{d}\right)^2 + O(H^3 (\log X)^3).
\end{split}
\end{equation}
The $k \neq 0$ summands of the first line contribute $S_2$ whereas the $k=0$ summand equals
\[
\begin{split}
&H \sum_{\substack{d_1, d_2 \leq D_0}}  a_{d_1} a_{d_2} \left(\sum_{\substack{m_1, m_2 \\ d_1 m_1 = d_2 m_2}} g\left(\frac{d_1 m_1}{X}\right) - \frac{\widehat{g}(0) X}{[d_1, d_2]}\right) \\
&= H \sum_{n} g\left(\frac{n}{X}\right) \left(\sum_{d \mid n} a_d\right)^2 - \widehat{g}(0)HX  \sum_{d_1, d_2 \leq D_0} a_{d_1} a_{d_2} \frac{1}{X^{10}} \sum_{\substack{n \leq X^{10} \\ [d_1, d_2] \mid n}} 1 + O(1) \\
&= S_3 + O(1).
\end{split}
\]

Hence it suffices to show that the main term on the second line of~\eqref{eq:Smiddle} contributes $S_1$, i.e.
\begin{equation}
\label{eq:S1Claim1}
\begin{split}
&\widehat{g}(0) X \sum_{\substack{d_1, d_2 \leq D_0}} \frac{a_{d_1} a_{d_2}}{d_1 d_2} \left((d_1, d_2)\sum_{\substack{|k| \leq H \\ (d_1, d_2) \mid k}} (H- |k|) - H^2\right) = S_1.
\end{split}
\end{equation}
Here 
\[
\begin{split}
\sum_{\substack{|k| \leq H \\ (d_1, d_2) \mid k}} (H-|k|) &= H + 2 \sum_{\substack{1 \leq r \leq \lfloor H / (d_1, d_2) \rfloor }} (H- r(d_1, d_2)) \\
&= H + 2 \left\lfloor \frac{H}{(d_1, d_2)}\right\rfloor \frac{H-(d_1, d_2) + H - \left\lfloor \frac{H}{(d_1, d_2)} \right\rfloor (d_1, d_2)}{2} \\
&= H + \left\lfloor \frac{H}{(d_1, d_2)}\right\rfloor \left(2H-(d_1, d_2)- \left\lfloor \frac{H}{(d_1, d_2)} \right\rfloor (d_1, d_2)\right).
\end{split}
\]
Writing $\theta_{d_1, d_2} := \frac{H}{(d_1, d_2)} - \lfloor \frac{H}{(d_1, d_2)} \rfloor$, this equals
\[
H + \left(\frac{H}{(d_1, d_2)} - \theta_{d_1, d_2}\right) (H-(d_1, d_2) + \theta_{d_1, d_2}(d_1, d_2)) = \frac{H^2}{(d_1, d_2)} + (d_1, d_2) \theta_{d_1, d_2}(1-\theta_{d_1, d_2}),
\]
so~\eqref{eq:S1Claim1} reduces to the claim
\begin{equation}
\label{eq:RMainTerm}
\widehat{g}(0) X \sum_{\substack{d_1, d_2 \leq D_0}} \frac{a_{d_1} a_{d_2}}{d_1 d_2}(d_1, d_2)^2 \theta_{d_1, d_2}(1-\theta_{d_1, d_2})= S_1.
\end{equation}
Writing $\psi(x)$ for the one-periodic function which is $x(1-x)$ for $x \in [0, 1]$ we see that
$\theta_{d_1, d_2}(1-\theta_{d_1, d_2}) = \psi(H/(d_1, d_2))$. It is easy to see that we have the Fourier expansion
\[
\begin{split}
\psi(x) &= \frac{1}{6} - \frac{1}{2\pi^2} \sum_{k \neq 0} \frac{1}{k^2} e(kx) = \frac{1}{6} - \frac{1}{\pi^2} \sum_{k \geq 1} \frac{\cos(2\pi k x)}{k^2} \\
&= \frac{1}{6} - \frac{1}{\pi^2} \sum_{k \geq 1} \frac{1-2\sin(\pi k x)^2}{k^2} = \frac{2}{\pi^2} \sum_{k \geq 1} \frac{\sin(\pi k x)^2}{k^2}.
\end{split}
\]
Hence the left hand side of~\eqref{eq:RMainTerm} equals
\[
2\widehat{g}(0) X \sum_{\substack{d_1, d_2 \leq D_0}} \frac{a_{d_1} a_{d_2}}{d_1 d_2} \sum_{k \geq 1} \left(\frac{(d_1, d_2)}{\pi k} \sin\left(\frac{\pi k H}{(d_1, d_2)}\right)\right)^2.
\]
Writing $k = e m$ with $e = (k, (d_1, d_2))$, this equals
\[
2\widehat{g}(0) X  \sum_{\substack{d_1, d_2 \leq D_0}} \frac{a_{d_1} a_{d_2}}{d_1 d_2} \sum_{\substack{e \geq 1 \\ e \mid (d_1, d_2)}} \sum_{\substack{m \geq 1 \\ (m, (d_1, d_2)/e) = 1}} \left(\frac{(d_1, d_2)}{\pi em} \sin\left(\frac{\pi em H}{(d_1, d_2)}\right)\right)^2.
\]
Substituting $d = (d_1, d_2)/e$, this is
\[
\begin{split}
&2\widehat{g}(0)X \sum_{\substack{d_1, d_2 \leq D_0}} \frac{a_{d_1} a_{d_2}}{d_1 d_2} \sum_{\substack{d \geq 1 \\ d \mid (d_1, d_2)}} \sum_{\substack{m \geq 1 \\ (m, d) = 1}} \left(\frac{d}{\pi m} \sin\left(\frac{\pi m H}{d}\right)\right)^2 \\
&= 2\widehat{g}(0)X \sum_{d \leq D_0} \gamma_{d, H} \sum_{\substack{d_1, d_2 \leq D_0 \\ d \mid (d_1, d_2)}} \frac{a_{d_1} a_{d_2}}{d_1 d_2} = S_1,
\end{split}
\]
and~\eqref{eq:RMainTerm} follows.
\end{proof}

In order to estimate $S_2$ we shall use Lemmas~\ref{le:KloTypeII} and~\ref{le:KloTypeI} that are consequences of the work of Deshouillers and Iwaniec~\cite{DI} on averages of Kloosterman sums. The following two lemmas and their proofs have very much in common with~\cite[Theorems 5 and 7]{BFI1} and~\cite{DIzeta1, DIzeta2}. Note that~\cite[Theorem 5]{BFI1} was used in a similar context in~\cite{MatoSmooth} whereas results from~\cite{DIzeta1, DIzeta2} have been used in studying almost all intervals of length $X^\theta$ (see e.g.~\cite{Jia}).

It will suffice to study $S_2$ with $a_d$ replaced by a type II sequence (i.e. $a_d = \sum_{d = mn} \alpha_m \beta_n$ for some complex coefficients $\alpha_m, \beta_n$ supported on certain ranges) and with $a_d$ replaced by a type I sequence (i.e. $a_d = \sum_{d = mn} \alpha_m$ with $\alpha_m$ supported on small and medium sized $m$).
 
In type II case we shall use the following lemma.
\begin{lemma}
\label{le:TypeIISpec}
Let $2 \leq H \leq X^{1/60}$, and let $g$ be a smooth compactly supported function satisfying~\eqref{eq:gderivbound}. Let $\alpha_m, \beta_n$ and $\gamma_q$ be bounded complex coefficients and $M, N, Q \geq 1$. Assume that 
\begin{equation}
\label{eq:typeIIcondSpec}
N \leq M \ll X^{21/50} \quad \text{and} \quad \max\{MN, Q\} \ll X^{14/25}.
\end{equation}
Then
\begin{equation}
\label{eq:typeIIBoundSpec}
\begin{split}
&\sum_{0 < |k| \leq H} \left| \sum_{\substack{m \sim M \\ n \sim N}} \alpha_m \beta_n \sum_{\substack{q \sim Q \\ (mn, q) \mid k}} \gamma_q \left(\sum_{\substack{\ell \\ \ell mn = k \pmod{q}}} g\left(\frac{\ell mn}{X}\right) - \frac{\widehat{g}(0) X}{[mn, q]}\right)\right| \ll X^{1-1/900}.
\end{split}
\end{equation}
\end{lemma}
This will readily follow from the following bound.
\begin{lemma}
\label{le:TypeII}
Let $\varepsilon >0$, let $X \geq H \geq 2$, and let $g$ be a smooth compactly supported function satisfying~\eqref{eq:gderivbound}. Let $\alpha_m, \beta_n$ and $\gamma_q$ be bounded complex coefficients and $M, N, Q \geq 1$. Then
\begin{equation}
\label{eq:typeIIClaim}
\begin{split}
&\sum_{0 < |k| \leq H} \left| \sum_{\substack{m \sim M \\ n \sim N}} \alpha_m \beta_n \sum_{\substack{q \sim Q \\ (mn, q) \mid k}} \gamma_q \left(\sum_{\substack{\ell \\ \ell mn = k \pmod{q}}} g\left(\frac{\ell mn}{X}\right) - \frac{\widehat{g}(0) X}{[mn, q]}\right)\right| \\ 
&\ll H^{1/2} X^{1/2+\varepsilon/10}  \Biggl[(MQ)^2 + \left(\frac{HMNQ}{X} + N\right) \\
&\qquad \qquad \qquad \cdot \left[MQ \left(\frac{HMNQ}{X} + N \right) (Q + N^2) + \frac{H(MN)^3Q}{X}\right] \Biggr]^{1/4}.
\end{split}
\end{equation}
\end{lemma}
\begin{proof}[Proof of Lemma~\ref{le:TypeIISpec} assuming Lemma~\ref{le:TypeII}]
Writing $W = \max\{MN, Q\}$ and noticing that $N^2 \leq MN \leq W$, we see that
\[
\begin{split}
&\left(\frac{HMNQ}{X} + N\right) \left[MQ \left(\frac{HMNQ}{X} + N \right) (Q + N^2) + \frac{H(MN)^3Q}{X}\right] \\
&\ll \left(\frac{HW^2}{X} + N\right) \left[\left(\frac{HW^2}{X} + N \right) MW^2 + \frac{HW^4}{X}\right] \ll \left(\frac{HW^2}{X} + N\right)^2 MW^2 \\
&\ll \frac{H^2 M W^6}{X^2} + N W^3 \ll H^2 X^{89/50} + W^{7/2} \ll H^2 X^{89/50} + X^{49/25}. 
\end{split}
\]
Hence Lemma~\ref{le:TypeII} implies that the left hand side of~\eqref{eq:typeIIBoundSpec} is
\[
\ll H^{1/2} X^{1/2+1/10000} (X^{49/100} + H^{1/2} X^{89/200})
\]
and the claim follows since $H \leq X^{1/60}$.
\end{proof}
\begin{proof}[Proof of Lemma~\ref{le:TypeII}]
Writing $\delta = (mn, q)$, the left hand side of~\eqref{eq:typeIIClaim} is at most
\[
\begin{split}
&\sum_{0 < \delta \leq H} \sum_{0 < |k| \leq H/\delta} \left| \sum_{\substack{m \sim M \\ n \sim N \\ \delta \mid mn}} \alpha_m \beta_n \sum_{\substack{c \sim Q/\delta \\ (c, mn/\delta) = 1}} \gamma_{\delta c} \left(\sum_{\substack{\ell \\ \ell \frac{mn}{\delta} \equiv k \pmod{c}}} g\left(\frac{\ell mn}{X}\right) - \frac{\widehat{g}(0) X}{cmn}\right)\right|.
\end{split}
\]
The summation condition can be rewritten as $\ell \equiv k \overline{\frac{mn}{\delta}} \pmod{c}$. Thus, by Poisson summation (Lemma~\ref{le:Poisson}(ii)), this equals
\begin{equation}
\label{eq:typeIIAfterPois}
\sum_{0 < \delta \leq H} \sum_{0 < |k| \leq H/\delta} \left| \sum_{\substack{m \sim M \\ n \sim N \\ \delta \mid mn}} \alpha_m \beta_n \sum_{\substack{c \sim Q/\delta \\ (c, mn/\delta) = 1}} \gamma_{\delta c} \frac{X}{cmn} \sum_{\substack{\ell \in \mathbb{Z} \\ \ell \neq 0}} \widehat{g}\left(\frac{\ell X}{cmn}\right) e\left(\frac{k \ell \overline{mn/\delta}}{c}\right)\right|.
\end{equation}
By~\eqref{eq:gFourbound} the contribution of $|\ell| > \frac{MNQ}{\delta X^{1-\varepsilon/20}}$ is
\[
\ll H X (\log X)^4 \sum_{|\ell| > \frac{MNQ}{\delta X^{1-\varepsilon/20}}} \left(\frac{MNQ}{\delta X \ell}\right)^{200/\varepsilon} \ll X^{-5}.
\]  
We write in~\eqref{eq:typeIIAfterPois} $\mu = (m, \delta)$ and $\nu = \delta/\mu$ so that $m = \mu d$ for some $d \in \mathbb{N}$ with $(d, \nu) = 1$. Since $\delta \mid mn$, we must have $\nu \mid n$ and thus we can write $n = \nu v$ for some $v \in \mathbb{N}$. Then $mn/\delta = dv$ so that $(dv, c) = 1$. With this notation the part of~\eqref{eq:typeIIAfterPois} with $0 < |\ell| \leq \frac{MNQ}{\delta X^{1-\varepsilon/20}}$ equals, for certain bounded coefficients $c_{k, \delta}$,
\begin{equation}
\label{eq:typeIILHSmi}
\begin{split}
E_{II} &:= X \sum_{0 < \delta \leq H} \sum_{0 < |k| \leq H/\delta} c_{k, \delta} \sum_{\delta = \mu \nu} \sum_{\substack{c \sim Q/\delta}} \frac{\gamma_{\delta c}}{c} \sum_{\substack{d \sim M/\mu \\ (d, c\nu) = 1}} \frac{\alpha_{\mu d}}{\mu d} \sum_{\substack{v \sim N/\nu \\ (v, c) = 1}} \frac{\beta_{\nu v}}{\nu v} \\
& \qquad \qquad \cdot \sum_{0 < |\ell| \leq \frac{MNQ}{\delta X^{1-\varepsilon/20}}} \widehat{g}\left(\frac{\ell X}{\delta c d v }\right) e\left(\frac{k \ell \overline{d v}}{c}\right).
\end{split}
\end{equation}
We write $u = k\ell$, and, in order to separate the variables $u$ and $v$ from the remaining ones,
\[
\widehat{g}\left(\frac{\ell X}{\delta c d v}\right) = \int_{-\infty}^\infty g(\xi) e \left(-\frac{\ell X}{\delta c d v} \xi \right) d\xi = \frac{\delta c d}{X} \int_{-\infty}^\infty g\left(\xi \frac{\delta c d}{X}\right) e \left(-\frac{\ell}{v} \xi\right) d\xi.
\]

Hence
\[
E_{II} \ll \frac{X^{1+\varepsilon/60}}{MNQ} \sum_{\substack{\mu, \nu \\ 0 < \mu \nu \leq H}} \mu \nu \max_{\xi \asymp \frac{X\mu}{MQ}} \sum_{\substack{c \sim \frac{Q}{\mu \nu}}} \sum_{\substack{d \sim M/\mu \\ (d, c) = 1}} \left| \sum_{0 < |u| \leq \frac{H MNQ}{\mu^2\nu^2 X^{1-\varepsilon/20}}} \sum_{\substack{v \sim N/\nu \\ (v, c) = 1}}  \alpha_{\xi, \mu, \nu} (u, v) e\left(\frac{u \overline{d v}}{c}\right)\right|,
\]
where
\[
\alpha_{\xi, \mu, \nu} (u, v) := N \frac{\beta_{\nu v}}{\nu v} \frac{1}{X^{\varepsilon/60}} \sum_{\substack{u = k\ell \\ 0 < |k| \leq \frac{H}{\mu \nu} \\ 0 < |\ell| \leq \frac{MNQ}{\mu \nu X^{1-\varepsilon/20}}}} c_{k, \mu \nu}  e\left(-\frac{\ell}{v} \xi\right) \ll 1.
\]

Now, for each $\mu$ and $\nu$, we apply Lemma~\ref{le:KloTypeII} with 
\[
C = \frac{2Q}{\mu \nu}, \quad D = \frac{2M}{\mu}, \quad V = \frac{2N}{\nu}, \quad \text{and} \quad U = \frac{HMNQ}{\mu^2 \nu^2 X^{1-\varepsilon/20}}. 
\]
Notice that in this notation we have $CDUV = 8H(MNQ)^2/(\mu^4 \nu^4 X^{1-\varepsilon/20})$. Hence, by Lemma~\ref{le:KloTypeII},
\[
\begin{split}
E_{II} & \ll H^{1/2} X^{1/2+\varepsilon/20} \sum_{\substack{\mu, \nu \\ 0 < \mu \nu \leq H}} \frac{1}{\mu \nu} \Biggl[\left(\frac{MQ}{\mu^2 \nu}\right)^{1/2} + \left(\frac{HMNQ}{\mu^2 \nu^2 X^{1-\varepsilon/20}} + \frac{N}{\nu}\right)^{1/4} \\
& \qquad \cdot \left[\frac{MQ}{\mu^2 \nu} \left(\frac{HMNQ}{\mu^2 \nu^2 X^{1-\varepsilon/20}} + \frac{N}{\nu}\right) \left(\frac{Q}{\mu \nu} + \left(\frac{N}{\nu}\right)^2\right) + \frac{HMNQ}{\mu^2 \nu^2 X^{1-\varepsilon/20}} \left(\frac{MN}{\mu \nu}\right)^2\right]^{1/4}\Biggr].
\end{split}
\]
The sums over $\delta$ and $\mu$ clearly contribute $O(X^{\varepsilon/100})$. Hence the above is
\[
\begin{split}
&\ll H^{1/2} X^{1/2+\varepsilon/10}  \Biggl[(MQ)^{1/2} + \left(\frac{HMNQ}{X} + N\right)^{1/4} \\
&\qquad \cdot \left[MQ \left(\frac{HMNQ}{X} + N \right) (Q + N^2) + \frac{H(MN)^3Q}{X}\right]^{1/4} \Biggr],
\end{split}
\]
as claimed.
\end{proof}

In the type I case (i.e. when studying $S_2$ from Proposition~\ref{prop:MeanSquare} with $a_d$ replaced by a type I sequence $a_d = \sum_{d = mn} \alpha_m$ with $\alpha_m$ supported on small and medium sized $m$) we shall use the following lemma. 
\begin{lemma}
\label{le:TypeISpec}
Let $2 \leq H \leq X^{1/60}$, let $\alpha_m$ and $\gamma_r$ be bounded complex coefficients, and let $M, N, Q, R \geq 1/2$ be such that
\begin{equation}
\label{eq:TypeISpecCond}
\max\{MN, QR\} \ll X^{31/50} \quad \text{and} \quad \max\{M, R\} \ll X^{6/25}.
\end{equation}
Let $g_1, g_2, g_3 \colon (0, \infty) \to \mathbb{R}$ be smooth compactly supported functions such that~\eqref{eq:gderivbound} holds for $g = g_j$ for each $j \in \{1, 2, 3\}$. Then
\begin{equation} 
\label{eq:TypeIClaimSpec}
\begin{split}
&\sum_{0 < |k| \leq H} \Biggl| \sum_{\substack{m \sim M \\ n}} \alpha_m g_1\left(\frac{n}{N}\right) \sum_{\substack{q \\ r \sim R \\ (qr, mn) \mid k}} \gamma_r g_2\left(\frac{q}{Q}\right) \\
& \qquad \qquad \qquad \Biggl(\sum_{\substack{\ell \\ \ell mn = k \pmod{qr}}} g_3\left(\frac{\ell mn}{X}\right) - \frac{\widehat{g_3}(0) X}{[mn, qr]}\Biggr)\Biggr| \ll X^{1-1/900}.
\end{split}
\end{equation}
\end{lemma}
Again this follows from a more general result.
\begin{lemma}
\label{le:TypeI}
Let $\varepsilon > 0$, let $X \geq H \geq 2$, let $\alpha_m$ and $\gamma_r$ be bounded complex coefficients, and let $M, N, Q, R \geq 1/2$ be such that $Q \ll MN$. Let $g_1, g_2, g_3 \colon (0, \infty) \to \mathbb{R}$ be smooth compactly supported functions such that~\eqref{eq:gderivbound} holds for $g = g_j$ for each $j \in \{1, 2, 3\}$. Then
\begin{equation}
\label{eq:TypeIClaim}
\begin{split}
&\sum_{0 < |k| \leq H} \Biggl| \sum_{\substack{m \sim M \\ n}} \alpha_m g_1\left(\frac{n}{N}\right) \sum_{\substack{q \\ r \sim R \\ (qr, mn) \mid k}} \gamma_r g_2\left(\frac{q}{Q}\right) \Biggl(\sum_{\substack{\ell \\ \ell mn = k \pmod{qr}}} g_3\left(\frac{\ell mn}{X}\right) - \frac{\widehat{g_3}(0) X}{[mn, qr]}\Biggr)\Biggr| \\
&\ll H^{1/2} X^{1/2+\varepsilon} \Biggl[R \left(MR + \frac{HMNQR}{X}\right)M + Q R \sqrt{\left(MR + \frac{HMNQR}{X}\right) M} + \frac{HM^2N^2}{X}\Biggr]^{1/2}.
\end{split}
\end{equation}
\end{lemma}

\begin{proof}[Proof of Lemma~\ref{le:TypeISpec} assuming Lemma~\ref{le:TypeI}]
We consider two cases.

\textbf{Case 1 ($Q \leq MN$):} We write $W = \max\{MN, QR\}$ and $U = \max\{M, R\}$. Then
\[
\begin{split}
&R \left(MR + \frac{HMNQR}{X}\right)M + Q R \sqrt{\left(MR + \frac{HMNQR}{X}\right) M} + \frac{HM^2N^2}{X} \\
&\ll U^2\left(U^2 + \frac{HW^2}{X}\right) + W\sqrt{\left(U^2 + \frac{HW^2}{X}\right) U} + \frac{HW^2}{X} \\
&\ll U^4 + H\frac{U^2 W^2}{X} + U^{3/2} W + H^{1/2} \frac{U^{1/2} W^2}{X^{1/2}} \ll X^{49/50} + H X^{18/25} + H^{1/2} X^{43/50}. 
\end{split}
\]
Hence by Lemma~\ref{le:TypeI} the left hand side of~\eqref{eq:TypeIClaimSpec} is
\[
\ll H^{1/2} X^{1/2+ 1/10000} \left(X^{49/50} + H X^{18/25} + H^{1/2} X^{43/50}\right)^{1/2} 
\]
and the claim follows since $H \leq X^{1/60}$.

\textbf{Case 2 ($Q > MN$):} In this case we can interchange the roles of $M,N$ with those of $R, Q$ in the claim~\eqref{eq:TypeIClaimSpec} by writing $\ell mn \equiv k \pmod{qr}$ first as $\ell m n = k + \ell' qr$, then using
\[
g_3\left(\frac{\ell mn}{X}\right) = g_3\left(\frac{\ell' q r}{X}\right) + O\left(\frac{H}{X}\right)
\]
and finally re-writing $\ell m n = k + \ell' qr$ as $\ell' qr = - k \pmod{mn}$, so that our claim becomes
\[
\begin{split}
&\sum_{0 < |k| \leq H} \Biggl| \sum_{\substack{q \\ r \sim R}} \gamma_r g_2\left(\frac{q}{Q}\right) \sum_{\substack{m \sim M \\ n \\  (mn, qr) \mid k}} \alpha_m g_1\left(\frac{n}{N}\right) \\ & \qquad \qquad \Biggl(\sum_{\substack{\ell' \\ \ell' qr = -k \pmod{mn}}} g_3\left(\frac{\ell' q r}{X}\right) -  \frac{\widehat{g_3}(0) X}{[qr, mn]}\Biggr)\Biggr| \ll X^{1-1/900}.
\end{split}
\]
Since $N \leq QR$, this follows from Case 1 with the roles of $R, Q$ and $M,N$ interchanged.
\end{proof}

\begin{proof}[Proof of Lemma~\ref{le:TypeI}]
Arguing as in the beginning of the proof of Lemma~\ref{le:TypeII} (i.e. writing $\delta = (mn, qr)$ and using Poisson summation (Lemma~\ref{le:Poisson}(ii))), we see that the left hand side of~\eqref{eq:TypeIClaim} is
\[
\begin{split}
&\ll \sum_{0 < \delta \leq H} \sum_{0 < |k| \leq H/\delta} \Biggl| \sum_{\substack{m \sim M, n \in \mathbb{N} \\ \delta \mid mn}} \alpha_m g_1\left(\frac{n}{N}\right) \sum_{\substack{r \sim R, q \in \mathbb{N} \\ \delta \mid qr \\ (qr/\delta, mn/\delta) = 1}} \gamma_r g_2\left(\frac{q}{Q}\right) \\
& \quad \quad \quad \quad \cdot \frac{X}{mnqr/\delta} \sum_{\substack{\ell \in \mathbb{Z} \\ \ell \neq 0}} \widehat{g}_3\left(\frac{\ell X}{mnqr/\delta}\right) e\left(\frac{k \ell \overline{mn/\delta}}{qr/\delta}\right)\Biggr|.
\end{split}
\]
Similarly to the proof of Lemma~\ref{le:TypeII} we can truncate the innermost sum to $0 < |\ell| \leq \frac{MNQR}{\delta X^{1-\varepsilon/20}}$. Still following proof of Lemma~\ref{le:TypeII}, we write $\mu = (m, \delta)$, and $\nu = \delta/\mu$, so that $m = \mu d$ for some $d\in \mathbb{N}$ with $(d, \nu) = 1$ and $n = \nu v$ for some $v \in \mathbb{N}$. This time we also write $\mu' = (r, \delta)$ and $\nu' = \delta/\mu'$ so that $r = \mu' s$ for some $s \in \mathbb{N}$ with $(s, \nu') = 1$. Since $\delta \mid qr$ we must have $\nu' \mid q$ and thus we can write $q = \nu' c$ for some $c \in \mathbb{N}$. Now $qr/\delta = c s$ and $mn/\delta = dv$ and so $(cs, dv) = 1$. With this notation we are, instead of $E_{II}$ in~\eqref{eq:typeIILHSmi}, led to
\begin{equation}
\label{eq:TypeIPh1}
\begin{split}
E_I &:= X \sum_{0 < \delta \leq H} \sum_{0 < |k| \leq H/\delta} c_{k, \delta} \sum_{\delta = \mu \nu} \sum_{\delta = \mu' \nu'} \sum_{\substack{s \sim R/\mu' \\ (s, \nu') = 1}} \frac{\gamma_{\mu' s}}{s} \sum_{\substack{c}} \frac{g_2\left(\frac{\nu' c}{Q}\right)}{c} \\
& \quad \cdot  \sum_{\substack{d \sim M/\mu \\ (d, cs\nu) = 1}} \frac{\alpha_{\mu d}}{\mu d} \sum_{\substack{v \\ (v, cs) = 1}} \frac{g_1\left(\frac{\nu v}{N}\right)}{\nu v} \sum_{0 < |\ell| \leq \frac{MNQR}{\delta X^{1-\varepsilon/20}}} \widehat{g}_3\left(\frac{\ell X}{\delta d v c s}\right) e\left(\frac{k \ell \overline{d v}}{cs}\right)
\end{split}
\end{equation}
for certain bounded coefficients $c_{k, \delta}$.

We write $n = k \ell$ and
\[
\widehat{g}_3\left(\frac{\ell X}{\delta d v c s}\right) = \frac{cv}{X} \int_{-\infty}^\infty g_3\left(\xi \frac{cv}{X}\right) e\left(\xi\frac{\ell}{\delta s d}\right) d\xi,
\]
and define
\[
b_{\delta, \mu, \mu', \xi}(d, n, s) :=  \mathbf{1}_{(s, \delta/\mu')} \mathbf{1}_{(d, \delta/\mu) = 1} M \frac{\alpha_{\mu d}}{\mu d} \cdot R \frac{\gamma_{\mu' s}}{\mu' s} \cdot \frac{1}{X^{\varepsilon/60}} \sum_{\substack{k \ell = n \\ 0 < |k| \leq H/\delta \\ 0 < |\ell| \leq \frac{MNQR}{\delta X^{1-\varepsilon/20}}}}  e\left(\xi \frac{\ell}{\delta ds}\right)
\] 
and
\[
g_{\nu, \nu', \xi}(x_1, x_2) := g_2(x_1) g_1(x_2) g_3\left(\xi \frac{x_1 x_2}{X} \frac{NQ}{\nu \nu'}\right).
\]
Then 
\[
\begin{split}
E_I &\ll \frac{X^{1+\varepsilon/60}}{MNQR} \sum_{0 < \delta \leq H} \delta \sum_{\substack{\delta = \mu \nu}} \sum_{\delta = \mu' \nu'} \max_{\xi \asymp \frac{X \nu \nu'}{NQ}}  \\
& \Biggl|\sum_{\substack{d \sim M/\mu \\ s \sim R/\mu' \\ (d, s) = 1}} \sum_{0 < |n| \leq \frac{H MN QR}{\delta^2 X^{1-\varepsilon/20}}} b_{\delta, \mu, \mu', \xi}(d, n, s) \sum_{\substack{c, v \\ (dv, cs) = 1}} g_{\nu, \nu', \xi}\left(\frac{c}{Q/\nu'}, \frac{v}{N/\nu}\right) e\left(\frac{n \overline{d v}}{cs}\right)\Biggr|,
\end{split}
\]
Note that in the support of the sum $b_{\delta, \mu, \mu', \xi}(d, n, s) \ll 1$ and $g_{\nu, \nu', \xi}(x_1, x_2)$ is smooth and compactly supported, and satisfies~\eqref{eq:g2dimderbound}. 

Now, for each $\delta, \mu, \mu'$, we apply Lemma~\ref{le:KloTypeI} with $C = Q/\nu', D = N/\nu, N = HMNQR/(\delta^2 X^{1-\varepsilon/20}), R = M/\mu = M\nu/\delta$ and $S = R/\mu' = \nu' R/\delta$, obtaining the bound 
\[
\begin{split}
E_I &\ll \frac{H^{1/2} X^{1/2+\varepsilon/20}}{(NQ)^{1/2}} \sum_{0 < \delta \leq H} \frac{1}{\delta} \sum_{\substack{\nu, \nu' \mid \delta}} \nu^{1/2}\nu'^{1/2} \Biggl[\frac{QR}{\delta}\left(\frac{\nu \nu' MR}{\delta^2} + \frac{HMNQR}{\delta^2 X^{1-\varepsilon/20}}\right)\left(\frac{Q}{\nu'} + \frac{MN}{\delta}\right) \\
&\qquad + \frac{Q^2 N R}{\nu' \nu \delta} \sqrt{\left(\frac{\nu \nu' MR}{\delta^2} + \frac{HMNQR}{\delta^2 X^{1-\varepsilon/20}}\right) \frac{\nu M}{\delta}} + \frac{N^2}{\nu^2} \frac{HMNQR}{\delta^2 X^{1-\varepsilon/20}} \frac{\nu M}{\nu' R}\Biggr]^{1/2}.
\end{split}
\]
The sums over $\nu, \nu'$ and $\delta$ contribute $O(X^{\varepsilon/100})$ so that
\[
\begin{split}
E_I &\ll \frac{H^{1/2} X^{1/2+\varepsilon}}{(NQ)^{1/2}} \Biggl[QR\left(MR + \frac{HMNQR}{X}\right)\left(Q + MN\right) \\
&\qquad + Q^2 N R \sqrt{\left(MR + \frac{HMNQR}{X}\right) M} +\frac{HM^2N^3Q}{X}\Biggr]^{1/2}.
\end{split}
\] 
By assumption, $Q+MN \leq 2MN$ and the claim follows.
\end{proof}

\section{Mean squares of $E^\pm(x, y, z)$}
\label{se:E(x)}
The aim of this section is to prove~\eqref{eq:E(x)claim}. We write $a_d^- = \alpha_d^-$ with $\alpha_d$ as in~\eqref{eq:alpk-def} and 
\[
a_d^+ = \sum_{a \in \mathcal{I}} \sum_{d = pe} \sigma\left(\frac{p}{\sqrt{2}^a} \right) \left(1-\frac{\log p}{\log y}\right) \alpha_{e, a}^+.
\]
with $\mathcal{I}, \sigma(x),$ and $\alpha_{e, a}^+$ as in~\eqref{eq:Idef}, \eqref{eq:sigmadef}, and~\eqref{eq:alphaka+def}. Notice that $a_d^\pm$ are supported on $d \leq DE = X^{5/9+1/1000}$. With these definitions,
\[
E^\pm(x, y, z) = \sum_{\substack{d \leq DE, m \in \mathbb{N} \\ x-h \log X < dm \leq x}} a_d^\pm - h\log X \sum_{d \leq DE} \frac{a_d^\pm}{d}.
\]

Let $g \colon \mathbb{R} \to [0, 1]$ be a smooth function supported on $[1/4, 2]$ such that $g(x) = 1$ for $x \in [1/2, 1]$ and~\eqref{eq:gderivbound} holds. Then Proposition~\ref{prop:MeanSquare} gives
\[
\int_{X/2}^X |E^\pm(x, y, z)|^2 \ll |S_1^\pm| + |S_2^\pm| + |S_3^\pm| + h^3 (\log X)^6
\]
with $S_1^\pm, S_2^\pm, S_3^\pm$ as in Proposition~\ref{prop:MeanSquare} with $a_d = a_d^\pm$ and $H = h\log X$. In the next three subsections we show that $S_j^\pm \ll hX$ for $j = 1, 2, 3$. 
\subsection{Showing that $S_1^\pm \ll hX$}
Noticing that $\gamma_{d, h\log X} \ll d h\log X$, it suffices to show that
\begin{align}
\label{eq:alm-Claim}
&\sum_{d \leq DE} d \left( \sum_{\substack{m \leq DE \\ m \equiv 0 \pmod{d}}} \frac{\alpha^-_{m}}{m}\right)^2 \ll \frac{1}{\log X}
\end{align}
and
\begin{align}
\label{eq:alm+Claim}
&\sum_{d \leq DE} d \left(\sum_{a \in \mathcal{I}} \sum_p \sigma\left(\frac{p}{\sqrt{2}^a} \right) \left(1-\frac{\log p}{\log y}\right) \sum_{\substack{m \leq D_a E \\ mp \equiv 0 \pmod{d}}}  \frac{\alpha_{m, a}^+}{mp}\right)^2 \ll \frac{1}{\log X}.
\end{align}

Splitting the sum over $p$ in~\eqref{eq:alm+Claim} according to whether $p \mid d$ or not and applying the inequality $(x+y)^2 \leq 2x^2+2y^2$, we see that the left hand side of~\eqref{eq:alm+Claim} is
\[
\begin{split}
&\ll \sum_{d \leq DE} d \left(\sum_{a \in \mathcal{I}} \sum_{p \mid d} \sigma\left(\frac{p}{\sqrt{2}^a} \right) \left(1-\frac{\log p}{\log y}\right) \sum_{\substack{m \leq D_a E \\ m \equiv 0 \pmod{d/p}}}  \frac{\alpha_{m, a}^+}{mp}\right)^2 \\
& + \sum_{d \leq DE} d \left(\sum_{a \in \mathcal{I}} \sum_{p \nmid d} \sigma\left(\frac{p}{\sqrt{2}^a} \right) \left(1-\frac{\log p}{\log y}\right) \sum_{\substack{m \leq D_a E \\ m \equiv 0 \pmod{d}}}  \frac{\alpha_{m, a}^+}{mp}\right)^2 =: S^+_{1, 1} + S^+_{1, 2},
\end{split}
\]
say. 

Let us first consider $S^+_{1, 1}$. Applying the Cauchy-Schwarz inequality, we obtain
\[
S^+_{1, 1} \ll \sum_{d \leq DE} d \left(\sum_{a \in \mathcal{I}} \sum_{p \mid d} \sigma\left(\frac{p}{\sqrt{2}^a} \right) \left(\sum_{\substack{m \leq D_a E \\ m \equiv 0 \pmod{d/p}}}  \frac{\alpha_{m, a}^+}{mp}\right)^2 \cdot \sum_{a \in I} \sum_{p \mid d} \sigma\left(\frac{p}{\sqrt{2}^a}\right) \right).
\]
Recalling the support of $\sigma$ we see that 
\begin{equation}
\label{eq:sumoverap}
\sum_{a \in I} \sum_{p \mid d} \sigma\left(\frac{p}{\sqrt{2}^a}\right) \ll \sum_{\substack{p \mid d \\ z/4 \leq p \leq 2y}} 1 \ll 1
\end{equation}
and
\begin{equation}
\label{eq:apsum}
\sum_{a \in \mathcal{I}} \sum_{p} \sigma\left(\frac{p}{\sqrt{2}^a}\right) \frac{1}{p} \ll \sum_{z/4 \leq p \leq 2y} \frac{1}{p} \ll 1.
\end{equation}

Using~\eqref{eq:sumoverap} and rearranging, we see that
\[
\begin{split}
S^+_{1, 1} &\ll \sum_{a \in \mathcal{I}} \sum_{p} \sigma\left(\frac{p}{\sqrt{2}^a}\right)  \frac{1}{p^2} \sum_{\substack{d \leq D E \\ p \mid d}} d \left(\sum_{\substack{m \leq D_a E \\ m \equiv 0 \pmod{d/p}}}  \frac{\alpha_{m, a}^+}{m}\right)^2.
\end{split}
\]
Substituting $d = p d'$ and applying~\eqref{eq:apsum}, we obtain
\begin{equation}
\label{eq:S+11bound}
\begin{split}
S^+_{1, 1} &\ll \sum_{a \in \mathcal{I}} \sum_{p} \sigma\left(\frac{p}{\sqrt{2}^a}\right) \frac{1}{p}  \sum_{\substack{d' \leq D_a E}} d' \left(\sum_{\substack{m \leq D_a E \\ m \equiv 0 \pmod{d'}}}  \frac{\alpha_{m, a}^+}{m}\right)^2 \\
&\ll \max_{a \in \mathcal{I}} \sum_{\substack{d' \leq D_a E}} d' \left(\sum_{\substack{m \leq D_a E \\ m \equiv 0 \pmod{d'}}}  \frac{\alpha_{m, a}^+}{m}\right)^2.
\end{split}
\end{equation}

Let us now turn to $S^+_{1, 2}$. Applying the Cauchy-Schwarz inequality, we see that
\[
\begin{split}
S^+_{1, 2} &\ll \sum_{d \leq DE} d \left(\sum_{a \in \mathcal{I}} \sum_{p \nmid d} \sigma\left(\frac{p}{\sqrt{2}^a} \right) \frac{1}{p} \sum_{\substack{m \leq D_a E \\ m \equiv 0 \pmod{d}}}  \frac{\alpha_{m, a}^+}{m}\right)^2\\
& \ll \sum_{d \leq DE} d \left(\sum_{a \in \mathcal{I}} \sum_{p} \sigma\left(\frac{p}{\sqrt{2}^a} \right) \frac{1}{p} \left(\sum_{\substack{m \leq D_a E \\ m \equiv 0 \pmod{d}}}  \frac{\alpha_{m, a}^+}{m}\right)^2 \cdot \sum_{a \in I} \sum_{p} \sigma\left(\frac{p}{\sqrt{2}^a}\right) \frac{1}{p} \right).
\end{split}
\]

Using~\eqref{eq:apsum}, rearranging, and using~\eqref{eq:apsum} again, we see that
\[
\begin{split}
S^+_{1, 2} &\ll  \sum_{a \in \mathcal{I}} \sum_{p} \sigma\left(\frac{p}{\sqrt{2}^a} \right) \frac{1}{p} \sum_{d \leq DE} d \left(\sum_{\substack{m \leq D_a E \\ m \equiv 0 \pmod{d}}}  \frac{\alpha_{m, a}^+}{m}\right)^2 \\
& \ll \max_{a \in \mathcal{I}} \sum_{d \leq D_a E} d \left(\sum_{\substack{m \leq D_a E \\ m \equiv 0 \pmod{d}}} \frac{\alpha^+_{m, a}}{m}\right)^2.
\end{split}
\]
Combining this with~\eqref{eq:S+11bound} we see that~\eqref{eq:alm+Claim} reduces to showing that
\begin{equation}
\label{eq:alm+claimreduced}
\max_{a \in \mathcal{I}} \sum_{d \leq D_a E} d \left(\sum_{\substack{m \leq D_a E \\ m \equiv 0 \pmod{d}}} \frac{\alpha^+_{m, a}}{m}\right)^2 \ll \frac{1}{\log X},
\end{equation}
a claim very similar to~\eqref{eq:alm-Claim}. A similar more general claim will be encountered also in \cite{MRPartII}\footnote{In the first arXiv version of~\cite{MRPartII} we used different sieve weights and utilized an incorrect version of Lemma~\ref{le:OperaCorrected} below (see Remark~\ref{rem:OperaCorr} below), so one should look at a more recent version (which is not yet on arXiv)}. 

Let us consider~\eqref{eq:alm+claimreduced}. Note first that the definition of $\alpha_{m, a}^+$ in~\eqref{eq:alphaka+def} implies that, for any $a \in \mathcal{I}$,
\[
\begin{split}
&\sum_{d \leq D_a E} d \left(\sum_{\substack{m \leq D_a E \\ m \equiv 0 \pmod{d}}} \frac{\alpha^+_{m, a}}{m}\right)^2 \\
&= \sum_{d \leq D_a E} d \left( \sum_{\substack{m \leq D_a E \\ m \equiv 0 \pmod{d}}} \frac{\mathbf{1}_{m \mid P(z)} \lambda_{(m, P(w, z)), a}^+ \rho_{(m, P(w))}^+}{m}\right)^2 \\
&= \sum_{\substack{d_1 \leq E \\ d_1 \mid P(w)}} d_1  \left( \sum_{\substack{m_1 \leq E, m_1 \mid P(w) \\ m_1 \equiv 0 \pmod{d_1}}} \frac{\rho_{m_1}^+}{m_1}\right)^2 \sum_{\substack{d_2 \leq D_a \\  d_2 \mid P(w, z)}} d_2 \left( \sum_{\substack{m_2 \leq D_a, m_2 \mid P(w, z) \\ m_2 \equiv 0 \pmod{d_2}}} \frac{\lambda_{m_2, a}^+}{m_2} \right)^2
\end{split}
\]
Here the sum over $d_2$ is
\[
\ll \sum_{\substack{d_2 \leq D_a \\  d_2 \mid P(w, z)}} d_2  \left(\frac{1}{d_2} \sum_{\substack{m_2 \leq D_a/d_2 \\ m_2 \mid P(w, z)}} \frac{1}{m_2}\right)^2 \ll \left(\sum_{\substack{r \leq D_a \\  r \mid P(w, z)}} \frac{1}{r}\right)^3 \ll \prod_{w \leq p < z} \left(1 + \frac{1}{p}\right)^3 \ll 1.
\]

We can argue similarly with~\eqref{eq:alm-Claim} and thus, noting that the support of $\rho_m^\pm$ is contained in $[1, E]$, it suffices to show that
\begin{equation}
\label{eq:rhopmMSclaim}
\sum_{\substack{d \mid P(w)}} d  \left( \sum_{\substack{m \mid P(w) \\ m \equiv 0 \pmod{d}}} \frac{\rho_{m}^\pm}{m}\right)^2 \ll \frac{1}{\log X}.
\end{equation}
A similar claim was shown in~\cite{Fried2} and also in~\cite[Lemma 6.18]{Opera} though there is a slight mistake in the latter proof. For completeness, we provide a detailed proof here.

The starting point for proving~\eqref{eq:rhopmMSclaim} is the following lemma.
\begin{lemma}
\label{le:OperaCorrected}
Let $w \geq 1$, let $\lambda_d$ be complex numbers, and define $\theta_n := \sum_{d \mid n} \lambda_d$. Write
\[
W := \sum_{\substack{d \mid P(w)}} d  \left( \sum_{\substack{m \mid P(w) \\ m \equiv 0 \pmod{d}}} \frac{\lambda_{m}}{m}\right)^2.
\]
Then
\begin{equation}
\label{eq:OperaCorrectedClaim}
W \ll \prod_{p < w} \left(1-\frac{1}{p}\right) \sum_{b \mid P(w)} \frac{b}{\varphi(b)^2} \sum_{\substack{e_1, e_2 \mid P(w) \\ (e_1, e_2) = 1 \\ (e_1 e_2, b) = 1}} \frac{|\theta_{be_1} \theta_{be_2}|}{e_1 e_2 \varphi(e_1e_2)}.
\end{equation}
\end{lemma}

\begin{remark}
\label{rem:OperaCorr}
In~\cite[Proof of Lemma 6.18]{Opera} it is claimed that 
\[
W = \prod_{p < w} \left(1-\frac{1}{p}+\frac{1}{p^2}\right) \sum_{b \mid P(w)} \frac{\theta_{b}^2}{b}.
\]
However there is a mistake in the proof on the second line of the second display of \cite[page 76]{Opera}, where the condition $(b, k) = 1$ is missing. Taking this condition into account leads to non-diagonal contribution as in our lemma though in applications the non-diagonal contribution is easy to handle. Our proof actually shows the exact formula
\begin{equation}
\label{eq:OperaCorrectedClaimPrecise}
\begin{split}
W &= \prod_{p < w} \left(1-\frac{1}{p}+\frac{1}{p^2}\right) \sum_{b \mid P(w)} \frac{b}{\varphi(b)^2} \prod_{p \mid b} \left(1-\frac{2p-1}{p^3-p^2+p}\right) \\
& \qquad \qquad \cdot \sum_{\substack{e_1, e_2 \mid P(w) \\ (e_1, e_2) = 1 \\ (e_1 e_2, b) = 1}} \frac{\theta_{be_1} \theta_{be_2} (-1)^{\omega(e_1 e_2)}}{e_1 e_2 \varphi(e_1e_2)} \prod_{p \mid e_1 e_2} \left(1-\frac{1}{p^2-p+1}\right).
\end{split}
\end{equation}
\end{remark}
\begin{proof}[Proof of Lemma~\ref{le:OperaCorrected}]
We follow the argument in \cite[Proof of Lemma 6.18]{Opera}, correcting the issue mentioned in Remark~\ref{rem:OperaCorr}. Notice that, by M\"obius inversion, $\lambda_m = \sum_{m = ab} \mu(a) \theta_b$. Hence, for $d \mid P(w)$,
\[
\begin{split}
\sum_{\substack{ m \mid P(w) \\ m \equiv 0 \pmod{d}}} \frac{\lambda_m}{m} &= \sum_{\substack{ab \mid P(w) \\ d \mid ab}} \frac{\mu(a) \theta_b}{ab} = \sum_{\substack{b \mid P(w)}} \frac{\theta_b}{b} \sum_{\substack{a \mid P(w) \\ (a, b) = 1 \\ \frac{d}{(b, d)} \mid a}} \frac{\mu(a)}{a} \\
&= \sum_{\substack{b \mid P(w)}} \frac{\theta_b}{b} \frac{\mu(d/(b, d))}{d/(b, d)}  \sum_{\substack{a \mid P(w) \\ (a, bd/(b, d)) = 1}} \frac{\mu(a)}{a} \\
&=  \sum_{\substack{b \mid P(w)}} \frac{\theta_b}{b} \frac{\mu(d/(b, d))}{d/(b, d)} \prod_{p < w} \left(1-\frac{1}{p}\right) \prod_{p \mid \frac{bd}{(b, d)}} \left(1-\frac{1}{p}\right)^{-1} \\
&= \prod_{p < w} \left(1-\frac{1}{p}\right) \sum_{\substack{b \mid P(w)}} \frac{\theta_b}{b} \frac{\mu(d)\mu((b, d))}{d/(b, d)} \frac{bd/(b,d)}{\varphi(bd/(b, d))}\\
&= \frac{\mu(d)}{\varphi(d)} \prod_{p < w} \left(1-\frac{1}{p}\right) \sum_{\substack{b \mid P(w)}} \theta_b \mu((b, d))\frac{\varphi((b, d))}{\varphi(b)}.
\end{split}
\]
Consequently
\[
W = \prod_{p < w} \left(1-\frac{1}{p}\right)^2 \sum_{b_1, b_2 \mid P(w)} \frac{\theta_{b_1} \theta_{b_2}}{\varphi(b_1)\varphi(b_2)} \sum_{d \mid P(w)} d \frac{\mu((b_1, d))\mu((b_2, d))\varphi((b_1, d))\varphi((b_2, d))}{\varphi(d)^2}.
\]
Writing $b_j = b e_j$ with $b = (b_1, b_2)$, we get
\[
W = \prod_{p < w} \left(1-\frac{1}{p}\right)^2 \sum_{\substack{b, e_1, e_2 \mid P(w) \\ (e_1, e_2) = (b, e_1 e_2) = 1}} \frac{\theta_{be_1} \theta_{be_2}}{\varphi(b e_1)\varphi(b e_2)} \sum_{d \mid P(w)} d \frac{\mu((e_1e_2, d))\varphi((be_1, d))\varphi((be_2, d))}{\varphi(d)^2}.
\]
The summand in the $d$-sum is multiplicative, so, looking at the Euler product factors, the $d$-sum equals
\[
\begin{split}
& \prod_{\substack{p < w \\ p \nmid b e_1 e_2}} \left(1+\frac{p}{(p-1)^2}\right) \prod_{\substack{p \mid b}} \left(1 + p \right)  \prod_{\substack{p \mid e_1 e_2}} \left(1 - \frac{p}{p-1}\right) \\
&= \prod_{p < w} \left(1+\frac{p}{(p-1)^2}\right) \prod_{\substack{p \mid b}} \left(1 + p \right)\left(1+\frac{p}{(p-1)^2}\right)^{-1}  \prod_{\substack{p \mid e_1 e_2}} \left(\frac{-1}{p-1}\right) \left(1+\frac{p}{(p-1)^2}\right)^{-1} \\
&= (-1)^{\omega(e_1 e_2)} \prod_{p < w} \left(1+\frac{p}{(p-1)^2}\right) \prod_{\substack{p \mid b}} \left(\frac{(p+1)(p-1)^2}{p^2-p+1}\right)  \prod_{\substack{p \mid e_1 e_2}} \left(\frac{p-1}{p^2-p+1}\right).
\end{split}
\]
Here 
\[
\frac{(p+1)(p-1)^2}{p^2-p+1} = p \frac{p^3-p^2-p+1}{p^3-p^2+p} = p\left(1-\frac{2p-1}{p^3-p^2+p}\right) 
\]
and
\[
\frac{p-1}{p^2-p+1} = \frac{1}{p} \cdot \frac{p^2-p}{p^2-p+1} = \frac{1}{p} \cdot \left(1 - \frac{1}{p^2-p+1}\right),
\]
so we get
\[
\begin{split}
W &= \prod_{p < w} \left(1-\frac{1}{p}\right)^2 \left(1+\frac{p}{(p-1)^2}\right) \sum_{b \mid P(w)} \frac{b}{\varphi(b)^2} \prod_{p \mid b} \left(1-\frac{2p-1}{p^3-p^2+p}\right) \\
&\cdot \sum_{\substack{e_1, e_2 \mid P(w) \\ (e_1, e_2) = (b, e_1 e_2) = 1}} \frac{\theta_{be_1} \theta_{be_2} (-1)^{\omega(e_1 e_2)}}{e_1 e_2\varphi(e_1 e_2)} \prod_{p \mid e_1 e_2} \left(1 - \frac{1}{p^2-p+1}\right).
\end{split}
\]
Here
\[
\left(1-\frac{1}{p}\right)^2 \left(1+\frac{p}{(p-1)^2}\right) = \frac{(p-1)^2}{p^2} \left(1+\frac{p}{(p-1)^2}\right) = \frac{(p-1)^2}{p^2} + \frac{1}{p} = 1-\frac{1}{p} + \frac{1}{p^2},
\]
so \eqref{eq:OperaCorrectedClaimPrecise} follows which implies also~\eqref{eq:OperaCorrectedClaim}.
\end{proof}

Let us now return to showing ~\eqref{eq:rhopmMSclaim}. For $r \geq 0$, write 
\[
V_r(n, w) := \sum_{\substack{n = p_1 \dotsm p_r d \\ p_r < p_{r-1} < \dotsc < p_1 < w \\ p \mid d \implies p \geq p_r \\ p_1 p_2 \dotsm p_r p_r^\beta \geq E \\ p_1 \dotsm p_h p_h^\beta < E \text{ for all odd $h < r$}}} 1
\]
By the definition of $\rho_e^\pm$, we have (see e.g. \cite[(6.29--6.30) with $g(p) = \mathbf{1}_{p \mid (n, P(w))}$]{IwaKov}), for any $n \in \mathbb{N}$,
\[
\begin{split}
\theta_n^+ &:= \sum_{e \mid n} \rho_e^+ = \sum_{e \mid (n, P(w))} \rho_e^+ = \mathbf{1}_{(n, P(w)) = 1} + \sum_{r \text{ odd}} V_r(n, w), \\
\theta_n^- &:= \sum_{e \mid n} \rho_e^- = \sum_{e \mid (n, P(w))} \rho_e^- = \mathbf{1}_{(n, P(w)) = 1} - \sum_{r \text{ even}} V_r(n, w).
\end{split}
\]

Recall that $\frac{\log E}{\log w} = \frac{1}{1000 \delta} \geq \beta = 30$ when $\delta$ is sufficiently small. One can easily show that, for every $r$, in the sum defining $V_r(n, w)$ one has 
\begin{equation}
\label{eq:wrdef}
p_r \geq w_r := w^{\left(\frac{\beta-1}{\beta}\right)^{r}}
\end{equation} 
(see e.g.~\cite[Section 6.3]{IwaKov}). In the support of $V_r(n, w)$ we have $\omega(n) \geq r$ so that $2^{\omega(n)-r} \geq 1$. Hence, writing $k = p_1 \dotsm p_r$, we have
\[
V_r(n, w) \leq 2^{\omega(n)-r} \mathbf{1}_{(n, P(w_r)) = 1} \sum_{\substack{n = kd \\ p \mid n \implies p \geq w_r}} 1 \leq 2^{-r} \mathbf{1}_{(n, P(w_r)) = 1} d(n)^2.
\]
Consequently
\begin{equation}
\label{eq:t(n)upperbound}
\left|\theta_n^+\right| = \left|\sum_{e \mid n} \rho_e^\pm\right| \leq \sum_{r \geq 0} 2^{-r} \mathbf{1}_{(n, P(w_r)) = 1} d(n)^2 =: \theta'_n,
\end{equation}
say. Clearly $|\theta_{b e_j}^\pm| \leq |\theta'_{b e_j}| \ll e_j^\varepsilon \theta_b'$. Plugging this into Lemma~\ref{le:OperaCorrected} and noticing that the sums over $e_1$ and $e_2$ are bounded  we obtain
\begin{equation}
\label{eq:dsumboundb}
\sum_{d \mid P(w)} d \left(\sum_{\substack{m \mid P(w) \\ m \equiv 0 \pmod{d}}} \frac{\rho_m^\pm}{m}\right)^2 \ll \prod_{p < w} \left(1-\frac{1}{p}\right) \sum_{b \mid P(w)} \frac{b}{\varphi(b)^2} \theta_b'^2 \ll \frac{1}{\log X} \sum_{b \mid P(w)} \frac{b}{\varphi(b)^2}\theta_b'^2.
\end{equation}
By the definition of $\theta_b'$ and the Cauchy-Schwarz inequality, 
\[
\theta_b'^2 \ll \sum_{r \geq 0} 2^{-r} \cdot \sum_{r \geq 0} 2^{-r} \mathbf{1}_{(b, P(w_r)) = 1} d(b)^4 \ll \sum_{r \geq 0} 2^{-r} \mathbf{1}_{(b, P(w_r)) = 1} d(b)^4.
\]
Hence
\[
\begin{split}
\sum_{b \mid P(w)} \frac{b}{\varphi(b)^2}\theta_b'^2 &\ll \sum_{r \geq 0} 2^{-r} \sum_{b \mid P(w)} \frac{b}{\varphi(b)^2} \mathbf{1}_{(b, P(w_r)) = 1} d(b)^4 \ll \sum_{r \geq 0} 2^{-r} \prod_{w_r \leq p < w} \left(1+\frac{16 p}{(p-1)^2}\right) \\
& \ll \sum_{r \geq 0} 2^{-r } \left(\frac{\log X}{\log w_r}\right)^{16} \ll \sum_{r \geq 0} 2^{-r } \left(\frac{\beta}{\beta-1}\right)^{16r} \ll 1
\end{split}
\]
since we chose $\beta = 30$ and $(30/29)^{16} < 2$. Now~\eqref{eq:rhopmMSclaim} follows from combining this with~\eqref{eq:dsumboundb}.

\subsection{Showing that $S_2^\pm \ll hX$} It suffices to establish that, for some small $\varepsilon > 0$ and any bounded $c_k$,
\begin{equation}
\label{eq:conv-Claim}
\sum_{0 < |k| \leq h\log X} c_k \sum_{\substack{d_1, d_2 \leq DE \\ (d_1, d_2) \mid k}} \alpha^-_{d_1} \alpha^-_{d_2} \left(\sum_{\substack{m_1, m_2 \\ d_1 m_1 = d_2 m_2 + k}} g\left(\frac{d_1 m_1}{X}\right) - \frac{\widehat{g}(0) X}{[d_1, d_2]}\right) \ll X^{1-\varepsilon/10}
\end{equation}
and
\begin{equation}
\label{eq:conv+Claim}
\begin{split}
&\sum_{0 < |k| \leq h\log X} c_k \sum_{a_1, a_2 \in \mathcal{I}} \sum_{p_1, p_2} \sigma\left(\frac{p_1}{\sqrt{2}^{a_1}} \right) \sigma\left(\frac{p_2}{\sqrt{2}^{a_2}} \right) \left(1-\frac{\log p_1}{\log y}\right) \left(1-\frac{\log p_2}{\log y}\right) \\
& \qquad \sum_{\substack{d_j \leq D_{a_j}E \\ (p_1 d_1, p_2 d_2) \mid k}} \alpha^+_{d_1, a_1} \alpha^+_{d_2, a_2} \left(\sum_{\substack{m_1, m_2 \\ d_1 p_1 m_1 = d_2 p_2 m_2+ k}} g\left(\frac{d_1 p_1 m_1}{X}\right) -  \frac{\widehat{g}(0) X}{[d_1 p_1, d_2p_2]}\right) \ll X^{1-\varepsilon/10}.
\end{split}
\end{equation}
These will follow from Lemmas~\ref{le:TypeIISpec} and~\ref{le:TypeISpec}.

Let us first consider~\eqref{eq:conv+Claim} which is more involved. It suffices to show that, for any $P_1, P_2 \in (z/4, 2y]$, any $D_i \leq DE/P_i$, and any bounded $\alpha_d, \beta_d$, one has
\begin{equation}
\label{eq:S2+claim}
\begin{split}
&\sum_{0 < |k| \leq h \log X}c_k \sum_{\substack{n_1, n_2}} \Lambda(n_1) \Lambda(n_2) h_1\left(\frac{n_1}{P_1}\right) h_2\left(\frac{n_2}{P_2}\right)\sum_{\substack{d_1 \sim D_1 \\ d_2 \sim D_2 \\ (d_1 n_1, d_2 n_2) \mid k}} \alpha_{d_1} \beta_{d_2} \\ & \qquad \qquad \left(\sum_{\substack{m_1, m_2 \\ d_1 m_1 n_1 = d_2 m_2 n_2+ k}} g\left(\frac{d_1 n_1 m_1}{X}\right)  -  \frac{\widehat{g}(0) X}{[d_1 n_1, d_2 n_2]}\right) \ll X^{1-\varepsilon/5}
\end{split}
\end{equation}
where 
\[
h_j(x) := \sigma(x) \left(1-\frac{\log (P_j x)}{\log y}\right) \frac{1}{\log (P_j x)} 
\]
are smooth, supported on $[1, 2]$ and satisfy
\begin{equation}
\label{eq:hjderbound}
\frac{d^k}{dx^k} h_j(x) \ll_k 1 \quad \text{for every $k \geq 0$.}
\end{equation}
In~\eqref{eq:S2+claim} we can write the condition $d_1 m_1 n_1 = d_2 m_2 n_2 + k$ as $d_1 m_1 n_1 \equiv k \pmod{d_2 n_2}$. Notice also that 
\begin{equation}
\label{eq:DjPjbound}
D_j P_j \leq DE \leq X^{14/25}
\end{equation}
We split into three cases according to the sizes of $P_i$.

\textbf{Case 1 ($P_1 \leq X^{21/50}$):} In this case we shall apply Lemma~\ref{le:TypeIISpec} with 
\[
N = \min\{D_1, P_1\}, \quad M = \max\{D_1, P_1\}, \quad \text{and} \quad Q \asymp D_2 P_2.
\]
We need to check that these choices satisfy~\eqref{eq:typeIIcondSpec}. By~\eqref{eq:DjPjbound} we have $\max\{MN, Q\} \ll X^{14/25}$, and by assumption $P_1 \leq X^{21/50}$. Hence it suffices to show that $D_1 \leq X^{21/50}$. But since $P_1 \geq z/4 = X^{5/36}/4$, we always have 
\[
D_1 \leq DE/P_1 \leq 4X^{5/9+1/1000-5/36} \leq X^{21/50}.
\]
Hence~\eqref{eq:S2+claim} follows from Lemma~\ref{le:TypeIISpec}; the choices of $\alpha_m$, $\beta_n$ are obvious and we can take
\[
\gamma_q = \frac{1}{X^{\varepsilon/100}} \sum_{q = d_2 n_2} \Lambda(n_2) h_2\left(\frac{n_2}{P_2}\right) \beta_{d_2}.
\]

\textbf{Case 2 ($P_2 \leq X^{21/50}$):}
Noting that 
\begin{equation}
\label{eq:swap}
g\left(\frac{d_1 n_1 m_1}{X}\right) = g\left(\frac{d_2 n_2 m_2}{X}\right) + O\left(\frac{h\log X}{X}\right)
\end{equation}
and that the summation condition $d_1 m_1 n_1 = d_2 m_2 n_2 + k$ can be written also as $d_2 m_2 n_2 \equiv -k \pmod{d_1 n_1}$, we obtain the claim similarly as in case $P_1 \leq X^{21/50}$, applying Lemma~\ref{le:TypeIISpec} with
\[
N = \min\{D_2, P_2\}, \quad M = \max\{D_2, P_2\}, \quad and \quad Q \asymp D_1 P_1.
\]

\textbf{Case 3 ($P_1, P_2 > X^{21/50}$):} Now $P_1, P_2 \in (X^{21/50}, 2y]$ and $D_i \leq DE/P_i$. In this case we apply Vaughan's identity (Lemma~\ref{le:Vaughan-identity}) to $n_1$ and $n_2$. Then it suffices to show that, with $P_j, D_j, \alpha_d, \beta_d, h_j(x)$ as in~\eqref{eq:S2+claim}, we have, for any bounded $c_k$,
\begin{equation}
\label{eq:multilinClaim}
\begin{split}
&\sum_{0 < |k| \leq h \log X} c_k \sum_{\substack{u_1, v_1 \\ u_2, v_2}} a_1(u_1) b_1(v_1) a_2(u_2) b_2(v_2) h_1\left(\frac{u_1 v_1}{P_1}\right) h_2\left(\frac{u_2 v_2}{P_2}\right) \sum_{\substack{d_1 \sim D_1 \\ d_2 \sim D_2 \\ (d_1 u_1 v_1, d_2 u_2 v_2) \mid k}} \alpha_{d_1} \beta_{d_2}  \\
& \Biggl(\sum_{\substack{m_1, m_2 \\ d_1 m_1 u_1 v_1 = d_2 m_2 u_2 v_2 + k}} g\left(\frac{d_1 u_1 v_1 m_1}{X}\right)  - \frac{\widehat{g}(0) X}{[d_1 u_1 v_1, d_2 u_2 v_2]}\Biggr) \ll X^{1-\varepsilon/4}
\end{split}
\end{equation}
whenever $a_j(u_j), b_j(v_j)$ for $j = 1, 2$, are bounded and such that
\begin{itemize}
\item $a_j(u_j)$ are supported on $(U_j, 2U_j] \subseteq (1/2, 2P_j^{1/2}]$ and $b_j(v_j)$ are supported on $(V_j, 2V_j] \subseteq (P_j^{1/2}/2, 4P_1]$. Moreover $U_j V_j \in (P_j/4, 2P_j]$ 
\item For each $j \in \{1, 2\}$ with $V_j \geq 4X^{2/3}$ one has $b_j(n) = \sigma(n/V_j)$ where $\sigma(x)$ is as in~\eqref{eq:sigmadef}. 
\end{itemize}
Notice that by~\eqref{eq:DjPjbound}
\begin{equation}
\label{eq:DjMjNjbound}
D_j U_j V_j \ll D_j P_j \ll DE \ll X^{14/25}.
\end{equation} 
We split further into three cases according to the sizes of $V_j$.

\textbf{Case 3.1 ($P_1, P_2 > X^{21/50}$ but $V_1 \leq X^{21/50}$):} In this case we shall apply Lemma~\ref{le:TypeIISpec} with 
\[
N \asymp \min\{D_1 U_1, V_1\}, \quad M \asymp \max\{D_1 U_1, V_1\}, \quad and \quad Q \asymp D_2 U_2 V_2.
\]
We need to check that these choices satisfy~\eqref{eq:typeIIcondSpec}. By~\eqref{eq:DjMjNjbound} we have $\max\{MN, Q\} \ll X^{14/25}$, and by assumption $V_1 \leq X^{21/50}$. Hence it suffices to show that $D_1 U_1 \ll X^{21/50}$. But since $P_1 > X^{21/50}$ and $V_1 \geq P_1^{1/2}/2$, we have by~\eqref{eq:DjMjNjbound} 
\[
D_1 \cdot U_1 \leq DE/V_1 \leq 2DE/P_1^{1/2} \ll X^{21/50}. 
\]
Hence~\eqref{eq:multilinClaim} follows from Lemma~\ref{le:TypeIISpec}; for instance
\[
\gamma_q = \frac{1}{X^{\varepsilon/100}} \sum_{\substack{q = d_2 u_2 v_2 \\ d_2 \sim D_2}} a_2(u_2) b_2(v_2) h_2\left(\frac{u_2 v_2}{P_2}\right) \beta_{d_2}.
\]

\textbf{Case 3.2 ($P_1, P_2 > X^{21/50}$ but $V_2 \leq X^{21/50}$):} Like Case 2 followed similarly to Case 1, this case follows similarly to Case 3.1, using a variant of~\eqref{eq:swap} and applying Lemma~\ref{le:TypeIISpec} with 
\[
N = \min\{D_2 U_2, V_2\}, \quad M = \max\{D_2 U_2, V_2\}, \quad \text{and} \quad Q = D_1 U_1 V_1.
\]

\textbf{Case 3.3 ($V_1, V_2 > X^{21/50}$):} In this case we have $b_j(v_j) = \sigma(v_j/V_j)$ for $j =1, 2$ and we shall apply Lemma~\ref{le:TypeISpec}. Before we can do this we need to separate the variables $u_j$ and $v_j$. Using the inverse Fourier transform we write
\[
h_j\left(\frac{u_j v_j}{P_j}\right) = \int_{-\infty}^\infty \widehat{h_j}(\xi_j) e\left(\frac{u_j v_j}{P_j} \xi_j\right) d\xi_j = \frac{1}{v_j} \int_{-\infty}^\infty \widehat{h_j}\left(\frac{\xi_j}{v_j}\right) e\left(\frac{u_j}{P_j} \xi_j\right) d\xi_j.
\]
Writing, for $j = 1, 2$, $a_{j, \xi_j}(u_j) := a_j(u_j) e(\frac{u_j}{P_j} \xi_j)$ and 
\[
g_{j, \xi}(x) := \frac{\sigma(x)}{x} \widehat{h_j}\left(\frac{\xi}{x V_j}\right) \left(1+\frac{|\xi|}{V_j}\right)^2,
\]
the claim~\eqref{eq:multilinClaim} reduces to the claim
\[
\begin{split}
&\int \int \sum_{0 < |k| \leq h \log X} c_k \sum_{\substack{u_1, v_1 \\ u_2, v_2}} a_{1, \xi_1}(u_1) a_{2, \xi_2}(u_2) g_{1, \xi_1}\left(\frac{v_1}{V_1}\right) g_{2, \xi_2}\left(\frac{v_2}{V_2}\right) \sum_{\substack{d_1 \sim D_1 \\ d_2 \sim D_2 \\ (d_1 u_1 v_1, d_2 u_2 v_2) \mid k}} \alpha_{d_1} \beta_{d_2} \\ 
&\Biggl(\sum_{\substack{m_1, m_2 \\ d_1 m_1 u_1 v_1 = d_2 m_2 u_2 v_2 + k}} g\left(\frac{d_1 u_1 v_1 m_1}{X}\right)  -  \frac{\widehat{g}(0) X}{[d_1 u_1 v_1, d_2 u_2 v_2]}\Biggr) \frac{d\xi_1 d\xi_2}{V_1V_2\left(1+\frac{|\xi_1|}{V_1}\right)^{2} \left(1+\frac{|\xi_2|}{V_2}\right)^{2}}   \ll X^{1-\varepsilon/3}
\end{split}
\]
with parameters as in~\eqref{eq:multilinClaim}. This follows once we have shown that, for any $\xi_1, \xi_2$, one has
\[
\begin{split}
&\sum_{0 < |k| \leq h \log X} c_k \sum_{\substack{u_1, v_1 \\ u_2, v_2}} a_{1, \xi_1}(u_1) a_{2, \xi_2}(u_2) g_{1, \xi_1}\left(\frac{v_1}{V_1}\right) g_{2, \xi_2}\left(\frac{v_2}{V_2}\right) \sum_{\substack{d_1 \sim D_1 \\ d_2 \sim D_2 \\ (d_1 u_1 v_1, d_2 u_2 v_2) \mid k}} \alpha_{d_1} \beta_{d_2} \\
& \quad \quad \Biggl(\sum_{\substack{m_1, m_2 \\ d_1 m_1 u_1 v_1 = d_2 m_2 u_2 v_2 + k}} g\left(\frac{d_1 u_1 v_1 m_1}{X}\right)  -  \frac{\widehat{g}(0) X}{[d_1 u_1 v_1, d_2 u_2 v_2]}\Biggr)  \ll X^{1-\varepsilon/3}.
\end{split}
\]
Using~\eqref{eq:sigmaderbound} and noting that derivatives of $\widehat{h_j}$ satisfy a variant of~\eqref{eq:gFourbound} thanks to~\eqref{eq:hjderbound}, one can show that~\eqref{eq:gderivbound} holds for $g = g_{j, \xi_j}$ for $j = 1, 2$.

Now we shall apply Lemma~\ref{le:TypeISpec} with
\[
M = D_1U_1, \quad N= V_1, \quad Q = V_2, \quad R = D_2U_2.
\]
We need to check that these choices satisfy~\eqref{eq:TypeISpecCond}.

By~\eqref{eq:DjMjNjbound} we have $\max\{MN, QR\} \ll X^{14/25} \ll X^{31/50}$ and thus it suffices to check that $D_j U_j \leq X^{6/25}$. But~\eqref{eq:DjMjNjbound} also implies that 
\[
D_j U_j \ll DE/V_j \ll X^{14/25-21/50} = X^{7/50}, 
\]
and hence the claim follows from Lemma~\ref{le:TypeISpec}.

Hence we have established~\eqref{eq:conv+Claim}. Let us now turn to the claim~\eqref{eq:conv-Claim}. Recall the definition of $\alpha_d^-$ from~\eqref{eq:alpk-def}. Using the well-factorability of the linear sieve weights (see~\cite[Section 12.7]{Opera}) we can find $k = O(1)$ and bounded coefficients $a_i^\pm(u)$ supported in $[1, X^{21/50}],$ and $b_i^\pm(v)$ supported on $[1, D/X^{21/50-\varepsilon}] = [1, X^{22/153 + \varepsilon}]$ such that, for every $d$,
\[
\lambda_d^\pm = \sum_{i=1}^k \sum_{d = uv} a^\pm_i(u) b^\pm_i(v)
\]
Using this and dyadic splitting, we see that~\eqref{eq:conv-Claim} follows once we have shown that, for any bounded coefficients $a(u), b(v)$ and any $U \leq X^{21/50}, V \leq X^{22/153+\varepsilon}$ and $D' \leq DE$ and $E' \leq E$, we have
\[
\begin{split}
&\sum_{0 < |k| \leq h \log X} \Biggl| \sum_{\substack{u \sim U \\ v \sim V \\ e \sim E'}} a(u) b(v) \rho^\pm_{e} \sum_{\substack{d_2 \sim D' \\ (u v e, d_2) \mid k}} \alpha_{d_2}^- \\ & \qquad \qquad \Biggl(\sum_{\substack{m_1 \\ uvem_1 \equiv k \pmod{d_2}}} g\left(\frac{u v e m_1}{X}\right)  - \widehat{g}(0) \frac{X}{[u v e, d_2]}\Biggr) \Biggr| \ll X^{1-\varepsilon/4}.
\end{split}
\]
But this follows from Lemma~\ref{le:TypeIISpec} with
\[
N = \min\{VE', U\}, \quad M = \max\{VE', U\}, \quad \text{and} \quad Q = D'.
\]

\subsection{Showing that $S_3^\pm \ll hX$}
Finally we need to show that, for $Y \in \{2X, X^{10}\}$, we have
\begin{equation}
\label{eq:S3-claim}
\sum_{n \leq Y} \left(\sum_{d \mid n} \alpha_d^-\right)^2 \ll \frac{Y}{\log X}
\end{equation}
and
\begin{equation}
\label{eq:S3+claim}
\sum_{n \leq Y} \left(\sum_{a \in \mathcal{I}} \sum_{p \mid n} \sigma\left(\frac{p}{\sqrt{2}^a} \right) \left(1-\frac{\log p}{\log y}\right) \sum_{\substack{d \mid P(z) \\ pd \mid n}} \alpha_{d, a}^+\right)^2 \ll \frac{Y}{\log X}.
\end{equation}
Here
\[
\sum_{\substack{d \mid P(z) \\ pd \mid n}} \alpha_{d, a}^+ = \sum_{\substack{d \mid P(w, z) \\ pd \mid n}} \lambda_{d,a}^+ \sum_{e \mid (n, P(w))} \rho_e^+ \ll \sum_{e \mid (n, P(w))} \rho_e^+.
\]
Using this and recalling~\eqref{eq:sumSigmaoverI} we see that the left hand side of~\eqref{eq:S3+claim} is
\[
\ll \sum_{n \leq Y} \left(\sum_{\substack{z/4 \leq p \leq 2y \\ p \mid n}} 1\right)^2 \left(\sum_{e \mid (n, P(w))} \rho_e^+\right)^2 \ll  \sum_{n \leq Y} \left(\sum_{e \mid (n, P(w))} \rho_e^+\right)^2.
\]
Hence~\eqref{eq:S3+claim} reduces to showing
\begin{equation}
\label{eq:S3+claim2}
\sum_{n \leq Y} \left(\sum_{e \mid (n, P(w))} \rho_e^+ \right)^2 \ll \frac{Y}{\log X}
\end{equation}
for $Y \in \{2X, X^{10}\}$. Similarly 
\[
\left|\sum_{d \mid n} \alpha_{d}^-\right| \ll \left|\sum_{e \mid (n, P(w))} \rho_e^+\right| + \left|\sum_{e \mid (n, P(w))} \rho_e^-\right|
\]
and thus~\eqref{eq:S3-claim} follows once we have shown that
\[
\sum_{n \leq Y} \left(\sum_{e \mid (n, P(w))} \rho_e^\pm \right)^2 \ll \frac{Y}{\log X}
\]
for $Y \in \{2X, X^{10}\}$.

Let us concentrate on showing~\eqref{eq:S3+claim2} for $Y = 2X$ as other claims follow in the same way. Recall~\eqref{eq:t(n)upperbound} and the definition of the parameter $w_r$ from~\eqref{eq:wrdef}. Using~\eqref{eq:t(n)upperbound} and applying the Cauchy-Schwarz inequality and the Shiu bound (Lemma~\ref{le:Shiu})
\[
\begin{split}
\sum_{n \leq 2X} \left(\sum_{e \mid (n, P(w))} \rho_e^+ \right)^2 &\ll \left(\sum_{r \geq 0} 2^{-r}\right) \cdot \left(\sum_{r \geq 0} 2^{-r} \sum_{n \leq 2X} \mathbf{1}_{(n, P(w_r)) = 1} d(n)^4\right) \\
&\ll \frac{X}{\log X} \sum_{r \geq 0} 2^{-r} \left( \frac{\beta}{\beta-1} \right)^{16 r} \ll \frac{X}{\log X}
\end{split}
\]
as claimed since $\beta = 30$ and $\left(\frac{30}{29}\right)^{16} < 2$.

\section*{Acknowledgments}
The author is greatful to John Friedlander and Henryk Iwaniec for discussions concerning~\cite[Chapter 6]{Opera}, to James Maynard and Maksym Radziwi{\l\l} for pointing out the possible alternative approach described in Remark~\ref{rem:Greaves}, and to Andrew Granville for pointing out Mikawa's work~\cite{Mikawa}. The author wishes to thank the referee for comments that helped to greatly improve the exposition of the paper. The author was supported by Academy of Finland grant no. 285894.
\bibliographystyle{plain}
\bibliography{AlmostPrimeBiblio}
\end{document}